\documentclass[12pt,reqno]{amsart}

\usepackage{times}
\usepackage{amsmath,amsfonts, amstext,amssymb,amsbsy,amsopn,amsthm}
\usepackage{dsfont}
\usepackage{esint}
\usepackage{graphicx}   
\usepackage{hyperref}
\usepackage[all]{xy}
\usepackage{color}
\usepackage{tikz}
\usepackage{lineno,hyperref}
\newcommand{\la}{\lambda}
\newcommand{\lap}{\mbox{$\bigtriangleup$}}

\newcommand{\ra}{{\mbox{$\rightarrow$}}}
\newcommand{\be}{\begin{equation}}
\newcommand{\ee}{\end{equation}}
\newtheorem{mthm}{Theorem}
\newtheorem{theorem}{Theorem}[section]
\newtheorem{lemma}[theorem]{Lemma}

\theoremstyle{definition}

\theoremstyle{remark}
\newtheorem{remark}{Remark}[section]

\theoremstyle{remark}

\numberwithin{equation}{section}
\allowdisplaybreaks[4]

\begin{document}

\title{Asymptotic method of moving planes for fractional  parabolic  equations}

\author{Wenxiong Chen }
\address{Department of Mathematical Sciences, Yeshiva University, New York, NY,  10033 USA}
\email{wchen@yu.edu}

\author{Pengyan Wang }
\address{School of   Mathematics and Statistics, Xinyang Normal University,
Xinyang,    464000, P.R. China, and Department of Mathematical Sciences, Yeshiva University, New York, NY,  10033 USA}
\email{wangpy119@126.com}

\author{Yahui Niu}
\address{ School of Mathematics and Statistics and Hubei Key Laboratory of Mathematical Sciences,
Central China Normal University, Wuhan, 430079, P.R. China, and Department of Mathematical Sciences, Yeshiva University, New York, NY,  10033 USA}
\email{yahuniu@163.com}

\author{Yunyun Hu}
\address{Department of Mathematical Sciences,
Yeshiva University, New York, NY, 10033 USA}
\email{yhu2@mail.yu.edu}

\date{\today}

\begin{abstract}
In this paper, we develop a systematical approach in applying an asymptotic method of moving planes to investigate qualitative properties of positive solutions for fractional parabolic equations. We first obtain a series of needed key ingredients such as  narrow region principles, and various asymptotic  maximum and strong maximum principles for antisymmetric functions in both bounded and unbounded domains. Then we illustrate how this new method can be employed to obtain asymptotic radial symmetry and monotonicity of positive solutions in a unit ball and on the whole space. Namely, we show that no matter what the initial data are, the solutions will eventually approach to radially symmetric functions.

We firmly believe that the ideas and methods introduced here can be conveniently applied to study a variety of nonlocal parabolic problems    with more general operators and more general nonlinearities.
\end{abstract}

\maketitle

\tableofcontents

\section{Introduction}
\label{s:introduction}

Reaction-diffusion equations and systems have been studied very extensively in the
past few years as models for many problems arising in applications such as, quasi-geostrophic flow \cite{CV}, general shadow and activator-inhibitor system \cite{Ni}, nonlocal diffusions \cite{BPSV,BV}, and the fractional porous medium \cite{PQRV}. During the last decade, elliptic
nonlocal equations, especially those involving fractional Laplacians, have been studied by more and more scholars and a series of results have been obtained, such as \cite{CLL} for symmetry and nonexistence of solutions, \cite{CLZ,CZ1} for Liouville Theorems with  indefinite  nonlinearities terms,
\cite{DPV} for existence
and symmetry results of   Schr\"{o}dinger equations, \cite{CS2014,RS} for regularity of solutions for Dirichlet problems, and    \cite{CC,CG,JMMX} for conformally covariant equations in conformal geometry.

It is well-known that symmetry and monotonicity of solutions play crucial  roles in the analysis of nonlinear PDEs and remain important topics in modern PDE theories. For elliptic equations involving either local or nonlocal operators, a number of systematic approaches have been established
to study these qualitative properties, such as the method of moving planes\cite{CH,DLL,DQ,L1,LZ1,LZ2,LyL,ZhL}, the method of moving spheres\cite{CLZr,LZm}, and the sliding methods \cite{CW3,DQW,LC,WC1,WC2}.
For parabolic equations, there have been some results for local operators. For instance, Li \cite{L2} obtained symmetry for positive solutions in the situations where the initial data are symmetric; Hess and  Pol\'{a}\u{c}ik \cite{HP} established asymptotic symmetry of positive solutions to the following parabolic equation in bounded domains
$$ \left\{\begin{array}{lll}
u_t=\lap u (x,t) +f(t,u(x,t) ),& x\in \Omega,~t>0,\\
u (x,t)=0,& x \in \partial\Omega,~t>0,\\
u(x,0)=u_0(x),&  x\in \Omega .
\end{array}\right. $$
Subsequently, Pol\'{a}\u{c}ik  made a number of progresses in such directions for parabolic equations involving local operators in both bounded and unbounded domains \cite{CP,P1,P2}.

However, for parabolic equations involving nonlocal operators, very little is known so far.
Recently, Jarohs and Weth \cite{JW1} established asymptotic symmetry of weak solutions for a class of nonlinear fractional reaction-diffusion equations in bounded domains.

In this paper, we will develop a systematical scheme  to carry on the asymptotic method of moving planes for nonlocal parabolic problems, either on bounded or unbounded domains. For nonlocal elliptic equations, these kinds of approaches were introduced in \cite{CLL} and then summarized in the book \cite{CLM}, among which the {\em narrow region principle} and the {\em decay at infinity} have been employed extensively by many researchers to solve various problems \cite{CQ,CLLy,WN}. A parallel system for the fractional parabolic equations  will be established here by very elementary methods, so that it can be conveniently applied to various nonlocal parabolic problems. As immediate applications, we consider the following fractional parabolic equation
\begin{equation}\label{eq:jj1}
\frac{\partial u}{\partial t}(x,t) +(-\lap)^s u(x,t)=f(t,u(x,t))
\end{equation}
in a unit ball and on the whole space.

For each fixed $t>0$,
\begin{equation*}\aligned
(-\lap)^s u(x,t)&=C_{N,s} P.V. \int_{\mathbb R^N} \frac{u(x,t)-u(y,t)}{|x-y|^{N+2s}}dy\\
&= C_{N,s}  \lim\limits_{\varepsilon\rightarrow 0}\int_{\mathbb R^N\setminus B_{\varepsilon}(x)}\frac{u(x,t)-u(y,t)}{|x-y|^{N+2s}}dy,
\endaligned \end{equation*}
where $0<s<1$ and $P.V.$ stands for the Cauchy principal value.

Define $$ {\mathcal L}_{2s}=\{u(\cdot, t) \in L^1_{loc} (\mathbb{R}^N) \mid \int_{\mathbb R^N} \frac{|u(x,t)|}{1+|x|^{N+2s}}dx<+\infty\},$$
then it is easy to see that for $u\in C^{1,1}_{loc}\cap {\mathcal L}_{2s}$, $(-\lap)^s u(x,t) $ is well defined.

We mainly study asymptotic symmetry and monotonicity of solutions, which shows a tendency of positive solutions of the parabolic equations without symmetric initial value to  ``improve their symmetry" as time increases, becoming ``symmetric and monotone in the limit" as $t\rightarrow \infty$. In this context, we  introduce  the $\omega$-limit set of $u$ in this space:
$$ \omega(u):=\{\varphi  \mid \varphi=\lim u(\cdot,t_k) ~\mbox{for~some}~t_k\rightarrow \infty\}.$$
By standard parabolic estimates \cite{L1996,P2}, $\omega(u)$ is a nonempty   compact  subset of $C_0(\mathbb R^N)$ and
$$ \underset{t\rightarrow \infty}{\lim} dist_{C_0(\mathbb R^N)} (u(\cdot,t), \omega(u))=0.$$
Here and below, $C_0(\mathbb R^N)$ stands for the space of continuous functions on $\mathbb R^N$ decaying to $0$ at infinity and we assume $C_0(\mathbb R^N)$ is equipped with the supremum norm.

In Section 2, we will establish a series of key principles, such as {\em asymptotic narrow region principles} and {\em asymptotic maximum principles near infinity}. Then in Section 3 and Section 4, we will use two examples to illustrate how these key ingredients obtained in Section 2 can be used in the  method of moving planes to establish asymptotic symmetry and monotonicity of the positive solutions.

\subsection{Main results on qualitative properties}

In Section 3, we consider the Dirichlet problem on a unit ball
\begin{equation}\label{eq:jj20}
\left\{\begin{array}{lll}
\frac{\partial u}{\partial t}(x,t)+(-\lap)^s u(x,t) =f(t,u(x,t)), &\qquad  (x,t)\in B_1(0) \times (0,\infty),\\
u(x,t) =0, &\qquad (x,t)\in B^c_1(0) \times (0,\infty),
\end{array}\right.
\end{equation}
where $B^c_1(0)$ is the complement of the unit ball in $\mathbb{R}^N$.

We prove
\begin{mthm}\label{thmj5}
Assume that $u(x,t)\in \Big(C^{1,1}_{loc} \big(B_1(0)\big) \cap C\big(\overline{B_1(0)}\big)\Big)\times C^1\big((0,\infty)\big)$ is a positive uniformly bounded solution of \eqref{eq:jj20}.

Let $\alpha\in(0,1)$  such that $\frac{\alpha}{2s}\in(0,1)$. Suppose $f(t,u)\in L^\infty(\mathbb{R}^+\times\mathbb{R})$  is $ C^{\frac{\alpha}{2s}}_{loc}$  in $t$,  Lipschitz continuous in $u$ uniformly with respect to $t$ and
\begin{equation}
\label{46}
f(t,0)\geq0,\ \ \ t\geq0.
\end{equation}

Then for each $\varphi(x) \in \omega(u)$, we have either
\smallskip

$\varphi(x)\equiv0$ or
\smallskip

$\varphi(x)$ are radially symmetric and strictly decreasing about the origin.
\end{mthm}

In Section 4, we study the nonlinear fractional parabolic equation on the whole space
\begin{equation}\label{eq:j1}
\frac{\partial u}{\partial t}+(-\lap)^s u=f(t,u),~(x,t) \in \mathbb R^N \times (0,\infty).
\end{equation}
We assume the following conditions on the nonlinearity $f$:\\
$(F):$ Let $\alpha\in(0,1)$  such that $\frac{\alpha}{2s}\in(0,1)$.  $f(t,u)$  is $ C^{\frac{\alpha}{2s}}_{loc}$  in $t$,  Lipschitz continuous in $u$ uniformly with respect to $t$
 and
\begin{equation}\label{eq:j103}
f(t,0)=0,\ \  \ f_u(t,0)<-\sigma,~t> 0 ,
  \end{equation}
where $\sigma>0$ is a constant. Also, assume that $f_u\equiv\frac{\partial f}{\partial u}$ is continuous near $u=0$.
\smallskip

We obtain
\begin{mthm}\label{thmj6}
Assume $f$ satisfies $(F)$. Let $u(x,t) \in \big( C^{1,1}_{loc}(\mathbb R^N)\cap {\mathcal L}_{2s} \big)\times C^1\big((0,\infty)\big)$ be a positive uniformly bounded solution of \eqref{eq:j1} satisfies
 \begin{equation}\label{eq:jj200}
\underset{|x|\rightarrow \infty }{\lim}~ u(x,t) =0,\ \ \text{uniformly for sufficiently large }t.
\end{equation}

Then we have, either
\smallskip

all $ \varphi (x) \in \omega (u)$ are identically 0 or
\smallskip

all $ \varphi (x) \in \omega (u)$ are radially symmetric and strictly decreasing about some point in $\mathbb R^N$.
\end{mthm}

When $f(t,u)=u^{p}(x,t)-u(x,t),~p>1$, \eqref{eq:j1} becomes the following  schr\"{o}dinger equation
$$ \frac{\partial u}{\partial t}+(-\lap)^s u(x,t)+u(x,t)=u^p(x,t).$$
Obviously,  $u^{p}(x,t)-u(x,t)$ satisfies \eqref{eq:j103} with small $u$.
Schr\"{o}dinger equations have received a great deal of interest from the mathematicians in the past twenty years or so, due in particular to their applications to optics. Schr\"{o}dinger equations also arise in quantum field theory, and in particular in the Hartree-Fock theory. For more details, one can see \cite{C2003,LK,SS} and the references therein.

\subsection{Ideas and key ingredients in the proofs}

To derive the radial symmetry and monotonicity of the functions in the $\omega$-limits set, we develop a series approaches in Section 2 to
carry on the {\em asymptotic method of moving planes} for nonlocal parabolic equations, which are quite different from the ones for nonlocal elliptic equations.

The key ingredients  and how they fit in the framework of the method of moving planes are illustrated in the following.

For simplicity, chose any direction to be the $x_1$ direction.  Let $$T_\lambda=\{x\in \mathbb R^N \mid x_1=\lambda,~\mbox{for}~\lambda \in \mathbb R\}$$
be the moving planes,
$$ \Sigma_\lambda= \{x\in \mathbb R^N \mid x_1< \lambda \}$$
be the region to the left of the plane, and
$$x^\lambda=(2\lambda-x_1, x_2,\cdots, x_N)$$
be the reflection of $x$ about the plane $T_\lambda$.

Assume that $u(x,t)$ is a solution of fractional parabolic equation \eqref{eq:jj1}. To compare the values of $u(x,t)$ with $u(x^\lambda,t)$, we denote
$$u_\lambda(x,t)= u(x^\lambda,t)~ \mbox{ and }~w_\lambda (x,t)=u_\lambda(x,t)-u(x,t).$$
Obviously, $w_\lambda (x,t)$ is anti-symmetric with respect to the plane $T_\lambda$, i.e.
\begin{equation*}
w_\lambda(x_1,x_2,\cdots,x_N,t)=-w_\lambda(2\lambda-x_1,x_2,\cdots,x_N,t).
\end{equation*}

For each $\varphi (x) \in \omega(u)$,
denote $$ \psi_\lambda(x)=\varphi(x^\lambda)-\varphi(x)=\varphi_\lambda(x)-\varphi(x),$$
which is an $\omega$-limit of $w_\lambda (x,t)$.

To obtain the symmetry of solutions in the unit ball, the first step is to show that for $\lambda$ sufficiently close to the left end of the domain, we have
\begin{equation}\label{eq:jj5}
\psi_\lambda(x)\geq 0,~x\in \Sigma_\lambda.
\end{equation}
This provides a starting position to move the plane.  The following {\em asymptotic  narrow region principle} will serve for this purpose.

\begin{mthm}\label{thmj1}(Asymptotic narrow region principle)
Let $\Omega$ be a region containing in the narrow slab $$\{x \mid \lambda-l<x_1<\lambda\}\subset \Sigma_\lambda$$  with some small $l$.

For $\bar{t}$ sufficiently large, assume that $ {w_\lambda(x,t)} \in (C^{1,1}_{loc}(\Omega) \cap {\mathcal L}_{2s}) \times C^1([\bar{t} ,\infty))$ is uniformly bounded and lower semi-continuous in $x$ on $\bar{\Omega}$, satisfying
\begin{equation*}
\left\{
\begin{array}{ll}
\frac{\partial w_\lambda}{\partial t}(x,t)+(-\triangle)^sw_\lambda(x,t)=c_\lambda(x,t) w_\lambda(x,t) ,~   &(x,t) \in  \Omega\times[\bar{t} ,\infty)  , \\
w_\lambda(x,t)\geq 0 , ~ &(x,t)  \in (\Sigma_\lambda  \backslash \Omega) \times[ \bar t,\infty), \\
w_\lambda(x,t)=- w_\lambda(x^\lambda,t), &(x,t)     \in \Omega\times[\bar{t} ,\infty) ,
\end{array}
\right.
\end{equation*}
where $c_\lambda(x,t)$ is bounded from above,  then the following statements hold:
\begin{enumerate}
\item[(\textbf{i})]   if $\Omega$ is bounded, then  for sufficiently small $l$,
\begin{equation}\label{eq:mb4}
\underset{t \rightarrow \infty}{\underline{\lim}}w_\lambda(x,t)\geq0 , \forall~ x\in \Omega;
\end{equation}
\item[(\textbf{ii})]  if $\Omega$ is unbounded, then the conclusion \eqref{eq:mb4} still holds under the condition
\begin{equation*}
\underset{|x|\rightarrow \infty}{\underline{\lim}}~w_\lambda(x,t)\geq 0, ~ \mbox{uniformly for} ~  t\geq \bar t .
\end{equation*}
\end{enumerate}
\end{mthm}

In the whole space $\mathbb{R}^N$, we start moving the plane from near either $-\infty$ or $\infty$,  and to this end, we employ the following

\begin{mthm}\label{1} (Asymptotic maximum principle near infinity)
Let $\Omega$ be an unbounded domain in $\Sigma_\lambda$.

For $\bar{t}$ sufficiently large, assume $w_\lambda\in (C^{1,1}_{loc}(\Omega)\cap {\mathcal L}_{2s} )\times C^1([\bar{t},\infty))$ is uniformly bounded and lower semi-continuous in $x$ on  $\bar{\Omega}$, satisfying
\begin{equation*}
\left\{
\begin{array}{ll}
    \frac{\partial w_\lambda}{\partial t}(x,t)+(-\triangle)^sw_\lambda(x,t)=c_\lambda(x,t)w_\lambda(x,t) ,~&   (x,t) \in  \Omega\times[\bar{t},\infty)  , \\
  w_\lambda(x,t)\geq 0 , &(x,t) \in  (\Sigma_\lambda  \backslash \Omega)\times[\bar{t},\infty),\\
  w_\lambda(x,t)=- w_\lambda(x^\lambda,t), &(x,t)\in \Omega\times[\bar{t},\infty).
\end{array}
\right.
\end{equation*}
Further assume that there is a constant $\sigma>0$ such that $c_\lambda(x,t)<-\sigma$ at the points in $\Omega\times[\bar{t},\infty)$ where $w_\lambda(x,t)<0$ and

\begin{equation*}
\underset{|x|\rightarrow \infty}{\underline{\lim}}~w_\lambda(x,t)\geq0,\ \text{ uniformly for } t\geqslant \bar{t}.
\end{equation*}

Then
\begin{equation*}
\underset{t \rightarrow \infty}{\underline{\lim}}w_\lambda(x,t)\geq0 , \forall~ x\in \Omega.
\end{equation*}
\end{mthm}

In the second step, we move the plane $T_\lambda$ continuously  to the right as long as inequality \eqref{eq:jj5} holds to its rightmost limiting position $T_{\lambda^-_0}$, where
$$\lambda_0^-=\sup\{\lambda \mid   \psi_\mu(x) \geq 0,~\text{for all } \varphi\in \omega(u),~\forall~x\in \Sigma_\mu,~\mu \leq \lambda\}. $$
In the case of the unit ball, we show that $\lambda^-_0$ must be zero. Then by the arbitrariness of the $x_1$ direction,
we deduced the radial symmetry and monotonicity of the solution about the origin. In the case of the whole space, one can only show that there is at least one $\varphi \in \omega(u)$ which is symmetric about the plane $T_{\lambda^-_0}$. In order to prove that all $\varphi \in \omega(u)$ must be symmetric about the same plane, we then start from near $x_1 = +\infty$ and move the plane to its left most position. More precisely,
define
$$\lambda^+_0 = \inf \{ \lambda \mid \psi_\mu (x) \leq 0, \, ~\text{for all }  \varphi\in \omega(u),\ x \in \Sigma_\mu, \, \mu \geq \lambda \}.$$

If $$\lambda^-_0 = \lambda^+_0,$$ we are done.

If $$\lambda^-_0 < \lambda^+_0,$$ then in the third step,
for some $\lambda$ in between, we will construct a sub-solution on $\mathbb{R}^N \setminus \Sigma_\lambda$ together with certain estimates on the asymptotic behavior of $w_\lambda (x,t)$ to derive a contradiction, which is one of the major tasks of this paper.

Such a general approach was introduced in \cite{P2} to deal with the local parabolic equations. However,  many of the techniques there
no longer work for nonlocal problems. Among others, one typical difference is that, a sub-solution for a local problem only need to be less or equal to the solution $w_\lambda (x,t)$ on the boundary $\partial D$ of the domain $D$, while for a nonlocal problem, this inequality needs to be hold on {\em the whole complement of the domain} $\mathbb{R}^N \setminus D$.
This makes the construction much more difficult, and to accomplish it, we introduce several new ideas, among which, one is to explore the anti-symmetry of $w_\lambda$ and make up an anti-symmetric sub-solution $\Psi(x,t)$. Based on the parabolic version of {\em maximum principle for anti-symmetric functions} established in Section 2, we only need to require that the inequality hold on the complement of $D$ in a half space:
$$ \Psi(x,t) \leq w_\lambda (x, t), \;\; x \in (\mathbb{R}^N\setminus\Sigma_\lambda) \setminus D.$$

We believe that these new ideas and techniques will become effective tools in solving a variety of other nonlocal and nonlinear parabolic problems.



The following strong maximum principles play important roles in the second and third steps.

\begin{mthm}\label{thmwj2} (Asymptotic strong maximum principle)
Assume $(F)$ holds. Suppose that $$u(x,t)\in (C^{1,1}_{loc}(\Omega)\cap {\mathcal L}_{2s} )\times C^1((0,\infty))$$  is a  positive uniformly bounded solution to
\begin{equation*}
\frac{\partial u}{\partial t}+(-\lap)^su=f(t,u),    (x,t)\in \Omega \times (0,\infty)
\end{equation*}
with \eqref{eq:jj200} holds,
where $\Omega\subset \mathbb R^N$ is either a bounded or an unbounded  domain.

Assume that there is some $\varphi \in \omega(u)$ which is positive somewhere in $\Omega$.
Then
$$\varphi(x)>0 \;  \mbox{ everywhere in } \; \Omega \; \mbox{  for all } \varphi \in \omega(u).$$

In other words, the following strong maximum principle holds for the whole family of functions $\{ \varphi \mid \varphi \in \omega(u)\}$ simultaneously:

Either
$$\varphi(x)\equiv 0 \;  \mbox{ everywhere in } \; \Omega \; \mbox{  for all } \varphi \in \omega(u),$$
or
$$\varphi(x)>0 \;  \mbox{ everywhere in } \; \Omega \; \mbox{  for all } \varphi \in \omega(u).$$

\end{mthm}

\begin{mthm}\label{thmj2} (Asymptotic strong maximum principle for antisymmetric function)
For sufficiently large $\bar{t}$, suppose that $w_\lambda(x,t)\in (C^{1,1}_{loc}(\Sigma_{\lambda})\cap {\mathcal L}_{2s} )\times C^1([\bar{t},\infty)) $ is bounded and satisfies
\begin{equation*}
\left\{\aligned
&\frac{\partial w_\lambda}{\partial t}+(-\Delta)^sw_\lambda=c_\la(x,t)w_\lambda,  & (x,t)\in \Sigma_{\lambda} \times [\bar{t},\infty),\\
&w_\lambda(x,t)=-w_\lambda(x^\lambda,t),  &(x,t)\in  \Sigma_{\lambda} \times [\bar{t},\infty),\\
&\underset{t \rightarrow \infty}{\underline{\lim}} w_\lambda(x,t)\geq 0, & x\in \Sigma_\la,
\endaligned \right.
\end{equation*}
with $c_\la(x,t)$ is bounded.
Assume $\psi_\lambda >0$ somewhere in $\Sigma_\lambda$.

Then $\psi_\lambda(x)>0$ in $\Sigma_{\lambda}$.
\end{mthm}

\begin{remark}\label{51}
As we can see from the proof of Theorem \ref{thmj2} later, the conclusion of Theorem \ref{thmj2} still holds when   $\Sigma_\la$ is replaced by any bounded domain in $\Sigma_\la$.
\end{remark}

\section{Proofs of key principles }
\label{s:maximum-principle}
In this section, we present the proofs of the key principles which will be used in establishing Theorems \ref {thmj5} and \ref {thmj6}.
For reader's convenience, we restate these theorems before their proofs.

\begin{theorem}\label{thmj1a}(Asymptotic narrow region principle)
Let $\Omega$ be a region containing in the narrow slab $$\{x \mid \lambda-l<x_1<\lambda\}\subset \Sigma_\lambda$$  with some small $l$.

For $\bar{t}$ sufficiently large, assume that $ {w_\lambda(x,t)} \in (C^{1,1}_{loc}(\Omega) \cap {\mathcal L}_{2s}) \times C^1([\bar{t} ,\infty))$ is uniformly bounded and lower semi-continuous in $x$ on $\bar{\Omega}$, satisfying
\begin{equation}\label{equationwa}
\left\{
\begin{array}{ll}
    \frac{\partial w_\lambda}{\partial t}(x,t)+(-\triangle)^sw_\lambda(x,t)=c_\lambda(x,t) w_\lambda(x,t) ,~   &(x,t) \in  \Omega\times[\bar{t} ,\infty)  , \\
  w_\lambda(x,t)\geq 0 , ~ &(x,t)  \in (\Sigma_\lambda  \backslash \Omega) \times[ \bar t,\infty), \\
  w_\lambda(x,t)=- w_\lambda(x^\lambda,t), &(x,t)     \in \Omega\times[\bar{t} ,\infty) ,
\end{array}
\right.
\end{equation}
where $c_\lambda(x,t)$ is bounded from above,  then the following statements hold:
\begin{enumerate}
\item[(\textbf{i})]   if $\Omega$ is bounded, then  for sufficiently small $l$,
\begin{equation}\label{eq:mb4a}
\underset{t \rightarrow \infty}{\underline{\lim}}w_\lambda(x,t)\geq0 , \forall~ x\in \Omega;
\end{equation}
\item[(\textbf{ii})]  if $\Omega$ is unbounded, then the conclusion \eqref{eq:mb4a} still holds under the condition
\begin{equation}\label{eq:mbbb4a}
\underset{|x|\rightarrow \infty}{\underline{\lim}}~w_\lambda(x,t)\geq 0, ~ \mbox{uniformly for} ~  t\geq \bar t .
\end{equation}
\end{enumerate}
\end{theorem}

\begin{proof}
[\bf Proof.]
($\textbf{i}$)
Let $m$ be a fixed positive constant to be chosen later. Set
$$\tilde{w}_\lambda(x,t)=e^{mt}w_\lambda(x,t),$$
then $\tilde{w}_\lambda(x,t)$ satisfies
\begin{equation}\label{ww}\aligned
\frac{\partial \tilde w_\lambda}{\partial t}(x,t)+(-\lap)^s \tilde w_\lambda(x,t)
&=\big(m+c_\la(x,t)\big) \tilde w_\lambda (x,t),~(x,t) \in \Omega \times[\bar{t} ,\infty).
\endaligned\end{equation}
We will establish \eqref{eq:mb4a} by proving, for any $T> \bar{t}$,
\begin{equation}\label{4}
\tilde w_\la (x,t) \geq \min\{0, \underset{  \Omega}{\inf}\tilde w_\la(x,\bar t) \},  ~(x,t) \in \Omega \times [\bar t,T].
\end{equation}

If (\ref{4}) is not true, by \eqref{equationwa} and the lower semi-continuity of $w_\lambda$ on $\bar {\Omega} \times [\bar t,T]$, there exists $(x_0, t_0)\in\Omega \times(\bar t,T]$ such that
\be\label{6}\tilde{w}_\lambda(x_0, t_0)=\min\limits_{\Sigma_\lambda \times(\bar t,T]}\tilde{w}_\lambda(x,t)<\min\{0,\inf_{\Omega}\tilde{w}_\lambda(x, \bar t)\},
\ee
then we have
\be \label{cwc20201}
\frac{\partial \tilde{w}_\lambda}{\partial t}(x_0, t_0)\leq 0
\ee
and
\begin{equation}
\label{laplace}
\begin{aligned}
&(-\triangle)^s\tilde{w}_\lambda(x_0, t_0)\\
=&C_{N,s}P.V.\int_{\Sigma_\lambda}\frac{\tilde{w}_\lambda(x_0, t_0)-\tilde{w}_\lambda(y, t_0)}{|x_0-y|^{N+2s}}dy
+C_{N,s}\int_{\mathbb R^N\backslash\Sigma_\lambda}\frac{\tilde{w}_\lambda(x_0, t_0)-\tilde{w}_\lambda(y, t_0)}{|x_0-y|^{N+2s}}dy\\
=&C_{N,s}P.V.\int_{\Sigma_\lambda}\frac{\tilde{w}_\lambda(x_0, t_0)-\tilde{w}_\lambda(y, t_0)}{|x_0-y|^{N+2s}}dy
+C_{N,s}\int_{\Sigma_\lambda}\frac{\tilde{w}_\lambda(x_0, t_0)-\tilde{w}_\lambda(y^\lambda, t_0)}{|x_0-y^\lambda|^{N+2s}}dy\\
\leq &C_{N,s}\int_{\Sigma_\lambda}\frac{2\tilde{w}_\lambda(x_0, t_0)}{|x_0-y^\lambda|^{N+2s}}dy=\frac{C}{l^{2s}}\tilde{w}_\lambda(x_0, t_0),
\end{aligned}
\end{equation}
where the last inequality holds because of
$|x_0 -y|<|x_0 -y^\lambda |$ in $\Sigma_\lambda$ and
if we set $$H=\big\{y=(y_1,y')\in\Sigma_\lambda\big|l<y_1-(x_0)_1<2l,\ |y'-x'_0|<l\big\},$$ then
$$
\int_{\Sigma_\lambda}\frac{1}{|x_0-y^\lambda|^{N+2s}}dy\geq\int_{H}\frac{1}{|x_0-y^\lambda|^{N+2s}}dy\geq\frac{C}{l^{N+2s}}|H|=\frac{C}{l^{2s}}.$$

As a result, we rewrite (\ref{ww}) and   obtain
\begin{equation}
\label{www}
\begin{aligned}
 \frac{\partial \tilde{w}_\lambda}{\partial t}(x_0,t_0)=&-(-\triangle)^s\tilde{w}_\lambda(x_0,t_0)+\left(m+c_\lambda(x_0,t_0)\right)\tilde{w}_\lambda(x_0,t_0)\\
\geq &\big(-\frac{C}{l^{2s}}+m+c_\lambda(x_0,t_0)\big)\tilde{w}_\lambda(x_0, t_0).
\end{aligned}
\end{equation}
Since $c_\la(x,t)$ is bounded from above for all $(x,t)$, we can first choose  $l$ small such that $$-\frac{C}{l^{2s}}+c_\lambda(x_0,t_0)<-\frac{C}{2l^{2s}}$$ and then take $m=\frac{C}{2l^{2s}}>0$
to derive that the right hand side of (\ref{www}) is strictly greater than $0$ since $\tilde{w}_\lambda(x_0, t_0)
<0.$ This contradicts \eqref{cwc20201}.

 Thus, by  the boundedness of $w_\lambda$,   there exists $C_1>0$ such that
\be\label{31}
\tilde{w}_\lambda(x, t)\geq \min\{0,\inf_{\Omega}\tilde{w}_\lambda(x, \bar t)\}\geq -C_1 , ~(x,t) \in  \Omega \times[\bar t,T], \; \forall \;T > \bar{t}.
\ee
Consequently, $$w_\lambda (x,t) \geq e^{-mt}(-C_1), \;\;  \forall \;t > \bar{t}.$$
Letting $t \ra \infty$, we arrive at
$$ \underset{t\rightarrow \infty}{\underline{\lim}}w_\lambda (x,t)\geq 0,~ x \in\Omega.$$

($\textbf{ii}$) When $\Omega$ is unbounded, \eqref{eq:mbbb4a} guarantees that the negative minimum of $\tilde w_\lambda(x,t)$  in   $\Sigma_\lambda \times [\bar t,T]$  must be attained at some finite point. Then one can follow the same discussion as the case of ($\textbf{i}$) to arrive at a contradiction.

This complete the proof of Theorem \ref{thmj1a}.
\end{proof}

\begin{theorem}\label{1a} (Asymptotic maximum principle near infinity)
Let $\Omega$ be an unbounded domain in $\Sigma_\lambda$.

For $\bar{t}$ sufficiently large, assume $w_\lambda\in (C^{1,1}_{loc}(\Omega)\cap {\mathcal L}_{2s} )\times C^1([\bar{t},\infty))$ is uniformly bounded and lower semi-continuous in $x$ on  $\bar{\Omega}$, satisfying
\begin{equation*}
\left\{
\begin{array}{ll}
    \frac{\partial w_\lambda}{\partial t}(x,t)+(-\triangle)^sw_\lambda(x,t)=c_\lambda(x,t)w_\lambda(x,t) ,~&   (x,t) \in  \Omega\times[\bar{t},\infty)  , \\
  w_\lambda(x,t)\geq 0 , &(x,t) \in  (\Sigma_\lambda  \backslash \Omega)\times[\bar{t},\infty),\\
  w_\lambda(x,t)=- w_\lambda(x^\lambda,t), &(x,t)\in \Omega\times[\bar{t},\infty) ,
\end{array}
\right.
\end{equation*}
Further assume that there is a constant $\sigma>0$ such that $c_\lambda(x,t)<-\sigma$ at the points in $\Omega\times[\bar{t},\infty)$ where $w_\lambda(x,t)<0$ and

\begin{equation*}
\underset{|x|\rightarrow \infty}{\underline{\lim}}~w_\lambda(x,t)\geq0,\ \text{ uniformly for } t\geqslant \bar{t}.
\end{equation*}

Then
\begin{equation*}
\underset{t \rightarrow \infty}{\underline{\lim}}w_\lambda(x,t)\geq0 , \forall~ x\in \Omega.
\end{equation*}
\end{theorem}

\begin{proof}
[\bf Proof.]

It is mostly similar to the proof of part (\textbf{ii}) in Theorem \ref{thmj1a} , here we only point out the differences.

By (\ref{laplace}), we have $$
(-\triangle)^s\tilde{w}_\lambda(x_0, t_0)\leq C_{N,s}\int_{\Sigma_\lambda}\frac{2\tilde{w}_\lambda(x_0, t_0)}{|x_0-y^\lambda|^{N+2s}}dy<0.
$$
Then analogous to (\ref{www}), we have
\begin{equation}
\label{7}
\begin{aligned}
 \frac{\partial \tilde{w}_\lambda}{\partial t}(x_0,t_0)=&-(-\triangle)^s\tilde{w}_\lambda(x_0,t_0)+\big(m+c_\lambda(x_0,t_0)\big)\tilde{w}_\lambda(x_0,t_0)\\
>&\big(m+c_\lambda(x_0,t_0)\big)\tilde{w}_\lambda(x_0, t_0).
\end{aligned}
\end{equation}
 Choosing $m=\frac{1}{2}\sigma>0$, then by $c_\lambda(x_0,t_0)<-\sigma$ and $\tilde{w}_\lambda(x_0, t_0)<0$, we have
\be\label{8}
\big(m+c_\lambda(x_0,t_0)\big)\tilde{w}_\lambda(x_0, t_0)>-\frac{1}{2}\sigma\tilde{w}_\lambda(x_0, t_0)>0.
\ee
Combining \eqref{7}, \eqref{8} with \eqref{cwc20201}, we obtain a contradiction.
\end{proof}

As we mentioned before, Theorems \ref{thmj1a} and \ref{1a} will provide us with a starting position in the first step to move the plane.
Then in the second step, we will need the following

\begin{theorem}\label{41a} (Asymptotic maximum principle in the union of a narrow region and the neighborhood of infinity)
Let $$
\Omega=\{x\in\Sigma_\lambda \mid dist\ (x, T_\lambda)<l\ \text{or } |x|>R\},
$$
where $l>0$ small and $R>0$ large.
Suppose $w_\lambda(x,t)$ satisfies the same conditions in Theorem \ref{1a}. Assume that $c_\lambda(x,t)$  is bounded from above in $\{x\in\Sigma_\lambda\mid dist\ (x, T_\lambda)<l\}$ and there exists $\sigma>0$ such that
$$c_\lambda(x,t)<-\sigma \; \mbox{ in } \{x\in\Sigma_\lambda \mid |x|>R\}.$$

Then we have
\begin{equation}\label{42a}
\underset{t \rightarrow \infty}{\underline{\lim}}w_\lambda(x,t)\geq0 , \forall~ x\in \Omega.\end{equation}
\end{theorem}

\begin{proof}[\bf Proof.]

It  can be obtained by a combination of the proofs of Theorem \ref{thmj1a} and Theorem \ref{1a} as the following.

For the minimum point $(x_0, t_0)\in\Omega \times(\bar t,T]$ in \eqref{6}, if $x_0\in\{x\in\Sigma_\lambda\mid dist\ (x, T_\lambda)<l\}$, then we proceed as in the proof of Theorem \ref{thmj1a} to obtain a contradiction; if $x_0\in\{x\in\Sigma_\lambda\mid |x|>R\}$, then by a similar process as in the proof of Theorem \ref{1a}, we can also obtain a contradiction and thus arrive at \eqref{42a}.
\end{proof}

\begin{theorem}\label{thmwj2a} (Asymptotic strong maximum principle)
Assume $(F)$ holds. Suppose that $$u(x,t)\in (C^{1,1}_{loc}(\Omega)\cap {\mathcal L}_{2s} )\times C^1((0,\infty))$$  is a  positive uniformly bounded solution to
\begin{equation*}
\frac{\partial u}{\partial t}+(-\lap)^su=f(t,u),    (x,t)\in \Omega \times (0,\infty)
\end{equation*}
with \eqref{eq:jj200} holds,
where $\Omega\subset \mathbb R^N$ is either a bounded or an unbounded  domain.

Assume that there is some $\varphi \in \omega(u)$ which is positive somewhere in $\Omega$.
Then
$$\varphi(x)>0 \;  \mbox{ everywhere in } \; \Omega \; \mbox{  for all } \varphi \in \omega(u).$$

In other words, the following strong maximum principle holds for the whole family of functions $\{ \varphi \mid \varphi \in \omega(u)\}$ simultaneously:

Either
$$\varphi(x)\equiv 0 \;  \mbox{ everywhere in } \; \Omega \; \mbox{  for all } \varphi \in \omega(u),$$
or
$$\varphi(x)>0 \;  \mbox{ everywhere in } \; \Omega \; \mbox{  for all } \varphi \in \omega(u).$$

\end{theorem}

In the proof of Theorem \ref{thmwj2a}, the following two lemmas play important roles.
\begin{lemma}\label{3} (Maximum principle)
Let $\Omega$ be a bounded domain in $\mathbb R^N$. Assume that $$v(x,t)\in (C^{1,1}_{loc}(\Omega)\cap {\mathcal L}_{2s} )\times C^1([0,\infty))$$  is lower semi-continuous in $x$ on $\bar{\Omega}$ and  satisfies
\begin{equation*}
\left\{\begin{array}{lll}
\frac{\partial v}{\partial t}+(-\lap)^sv\geq c(x,t)v,& (x,t)\in \Omega\times [0,+\infty),\\
v(x,t)\geq 0,& (x,t)\in (\mathbb R^N\backslash \Omega )\times [0  ,+\infty),\\
v(x,0)\geq 0,& x\in\Omega.
\end{array}\right.
\end{equation*}
with $c(x,t)$ bounded from above.
Then
 \be\label{9} v(x,t)\geq 0,~ \ \ \ (x,t)\in \Omega\times [0,T],\ \ \forall\ \ T>0.\ee
 \end{lemma}

\begin{remark}\label{52}
If $\Omega$ is unbounded, by assuming that
$$
\underset{|x|\rightarrow \infty }{\lim}~ v(x,t) \geq 0,\  \mbox{uniformly ~for }t>0 ,
$$
we can also obtain \eqref{9}.
\end{remark}
\begin{remark}
In fact, during the proof, we only need $c(x,t)$ be bounded from above at the points in $\Omega\times [0,+\infty)$ where $v(x,t)<0$.
\end{remark}

\begin{proof}[\bf Proof]
Since $c(x,t)$ is bounded from above, we can choose $m<0$ such that $m+c(x,t)<0$.
Set $$\tilde{v}(x,t)=e^{mt}v(x,t),$$ then $\tilde{v}(x,t)$ satisfies
\be\label{33}
\frac{\partial \tilde{v}}{\partial t}+(-\lap)^s \tilde{v}\geq (m+c(x,t)) \tilde{v},~(x,t) \in \Omega \times[0,\infty).
\ee
We claim that $$
\tilde{v}(x,t)\geq \inf_\Omega\tilde{v}(x,0)\geq 0\ \ \ \text{in}\ \ \Omega\times[0,T].
$$
If not, there exists $(x_0,t_0)\in\Omega\times (0,T]$ such that
$$
\tilde{v}(x_0,t_0)=\min_{\mathbb{R}^N\times[0,T]}\tilde{v}(x,t)<0,
$$
then
$$ (m+c(x_0,t_0))\tilde{v}(x_0, t_0)>0.$$
While,
$$\frac{\partial \tilde{v}}{\partial t}(x_0,t_0)\leq 0,
$$

$$
(-\triangle)^s\tilde{v}(x_0, t_0)
=C_{N,s}P.V.\int_{\mathbb{R}^N}\frac{\tilde{v}(x_0, t_0)-\tilde{v}(y, t_0)}{|x_0-y|^{N+2s}}dy< 0.
$$
This contradicts \eqref{33}. So \eqref{9} holds.
\end{proof}

\begin{lemma}\label{limit lemma}
Suppose $u$ is a solution of \eqref{eq:j1} satisfying \eqref{eq:jj200} and $(F)$.

Assume that $ \varphi(x) \not \equiv 0$ for some $\varphi\in \omega(u)$. Or equivalently assume that

\be \label{eq:4w2021}
\underset{t\rightarrow \infty}{\overline{\lim}}~\|u(x,t)\|_{L^\infty(\mathbb{R}^N)}>0.
\ee
Then we have
\be\label{eq:ww2020}
\underset{t\rightarrow\infty}{\underline{\lim}}~\|u(x,t)\|_{L^\infty(\mathbb{R}^N)}>0.
\ee
\end{lemma}
\begin{remark}
We will see from the proof below that Lemma \ref{limit lemma} still holds if we replace the $\mathbb R^N$ in \eqref{eq:4w2021} and \eqref{eq:ww2020} with any domain $\Omega\in \mathbb R^N$ bounded or unbounded.\end{remark}
\begin{proof}[\bf Proof]

Since $f_u(t,0)<-\sigma$ and $f_u$ is continuous near $u=0$, then there exists $\varepsilon_o>0$ small such that for any  $$0<\eta<\varepsilon_o,$$ we have
\be\label{35}
f_u(t,\eta)<-\sigma.
\ee
If \eqref{eq:ww2020} is false, then there exists $t_k\rightarrow \infty$ such that
$$\|u(x,t_k)\|_{L^\infty(\mathbb{R}^N)}\rightarrow 0.$$
Therefore, for the $\varepsilon_o$ above, there exists $k_0\in \mathbb{N}$ such that for any $k\geq  k_0$, we have
$$
u(x,t_k)<\varepsilon_o\ \ \ \text{in}\ \ \mathbb{R}^N.
$$

Now we fix such a $t_k(k\geq k_0)$ and  set
$$
u_k(x,t)=u(x,t+t_k).
$$
Then it satisfies
$$
\frac{\partial u_k}{\partial t}+(-\lap)^s u_k=f(t+t_k,u_k), \ (x,t)\in \mathbb{R}^N\times (0,\infty).$$
Next we will construct an upper solution to control $u_k$ in $\mathbb{R}^N\times (0,\infty)$.

Let $\xi(t)$ be a solution of
$$ \left\{\begin{array}{lll}
\frac{d \xi(t)}{dt}=f(t+t_k,\xi(t)),\ \ \ \ ~t\in (0,\infty),\\
\xi(0)=\varepsilon_o.
\end{array}\right.$$
Observing that
$$
f(t+t_k,\xi(t))=f_u\big(t+t_k,\gamma(t)\big)\xi(t),\ \ \gamma(t)\ \text{is between } 0 \text{ and } \xi(t),
$$
since $f(t,0)=0.$
Then, it is easy to see that
$$
\xi(t)=\varepsilon_o\ e^{\int_0^tf_u(\tau+t_k,\gamma(\tau))d\tau}.
$$
For $t=0$, by $\gamma(0)<\xi(0)=\varepsilon_o$ and \eqref{35}, we have
$$\frac{d \xi}{dt}(0)=f_u\big(t_k,\gamma(0)\big)\xi(0)<-\sigma\varepsilon_o.$$
Hence, $\xi(t)$ is decreasing near   $t=0$ and for $t>0$ small, we have $$0<\gamma(t)<\xi(t)<\varepsilon_oe^{-\sigma t}<\varepsilon_o.$$
By \eqref{35}, we can proceed this process continuously as $t$ increasing to derive that
$$
0<\xi(t)<\varepsilon_oe^{-\sigma t},\ \ \ \forall\ t>0.
$$
As a result, \be \label{eq:4w2020}
\xi(t)\searrow0,\ \ \ \text{as}~t\rightarrow \infty.
\ee

Now we compare $\xi(t)$ and $u_k(x,t)$ in $\mathbb{R}^N\times (0,\infty)$.
Let $$v(x,t)=\xi(t)-u_k(x,t).$$  Then
$$ \left\{\begin{array}{lll}
\frac{\partial v}{\partial t}+(-\lap)^sv=c_{\lambda,k}(x,t)v,& (x,t)\in \mathbb{R}^N\times (0,\infty),\\
v(x,0)\geq 0,& x\in \mathbb{R}^N,
\end{array}\right.$$
where $$c_{\lambda,k}(x,t)=\frac{f(t+t_k,\xi(t))-f(t+t_k,u_k(x,t))}{\xi(t)-u_k(x,t)}$$ is bounded since $f(t,u)$ is Lipschitz continuous in $u$ uniformly
with respect to $t$. Taking into account of Remark \ref{52}, we obtain
$$
v(x,t)\geq 0,~(x,t)\in \mathbb{R}^N\times [0,T],\ \ \ \forall\ T>0.$$
Especially,
$$\xi(T)\geq u(x,T),~x\in \mathbb{R}^N,\ \ \ \forall\ T>0.$$
Letting $T\rightarrow \infty$, we deduce from \eqref{eq:4w2020} that
$u(x,T)\rightarrow 0,~x\in \mathbb{R}^N.$ This   contradicts  \eqref{eq:4w2021} and thus establishes \eqref{eq:ww2020}.
\end{proof}

\begin{proof}[\bf Proof of Theorem \ref{thmwj2a}]
 From the definition of $\omega(u)$, for each  $\varphi\in \omega(u)$, there  exists a sequence $\{t_k\}$ such that
$$u(x,t_k)\rightarrow \varphi(x) \;\; \mbox{ as } \; t_k\rightarrow \infty.$$
Set $$ u_k(x,t)=u(x,t+t_k-1)$$
and
 $$ f_k(t,u)=f(t+t_k-1,u).$$
Then $u_k(x,1)\rightarrow \varphi(x)~\mbox{as}~k\rightarrow \infty.$
Since $f(t,u)$ is $C^{\frac{\alpha}{2s}}$ in $t$,  by regularity theory for parabolic equations \cite{XFR}, there exists some functions $u_\infty$ and $\tilde f $ such that $u_k\rightarrow u_\infty$,  $f_k\rightarrow \tilde f $.
And $u_\infty(x,t)$ satisfies
\begin{equation}
\label{49}
\frac{\partial u_\infty}{\partial t}+(-\lap)^su_\infty= \tilde f_u(t,\eta(x,t))u_\infty, \ \ (x,t)\in \Omega\times[0,2],
\end{equation}
where $\tilde{c}(x,t)=\frac{\tilde f(t,u_\infty)}{u_\infty}$ is is bounded by condition $(F)$.  Again by \cite{XFR}, $u_\infty(x,t)$ is at least H\"{o}lder continuous in both $t$ and $x$.

Take $m>0$ such that
$$
m+ \tilde{c}(x,t)>0.
$$ Denote $$\bar u(x,t)=e^{mt}u_\infty(x,t),$$ then
\be \label{12}\frac{\partial \bar u}{\partial t}+(-\lap)^s \bar u= \Big(m+ \tilde{c}(x,t)\Big)\bar u\geq 0,~(x,t)\in \Omega\times [0,2]. \ee

Since there is some $\varphi \in \omega(u)$ satisfies $\varphi>0$ somewhere in $\Omega$, then Lemma \ref{limit lemma} implies that
for all $\varphi \in \omega(u)$, $\varphi>0$ somewhere in $\Omega$.
By continuity, there exists a set $D\subset\subset \Omega$ such that
\begin{equation}\label{eq:4w01}
\varphi (x)\geq c_0 >0,\ \, \  x \in D,
\end{equation} where $c_0$ is a positive small constant. This means that
$$
u_\infty(x,1)\geq c_0 >0,\ \, \  x \in D.
$$

From the continuity of $u_\infty(x,t)$, there exists $0<\varepsilon_o <1$, such that
$$
u_\infty(x,t)\geq c_0/2 >0,\ \, \  (x,t) \in D\times[1-\varepsilon_o,1+\varepsilon_o].
$$
Consequently,
\begin{equation}\label{11}
\bar u(x,t)=e^{mt}u_\infty(x,t)\geq c_0/2 >0,\ \, \  (x,t) \in D\times [1-\varepsilon_o,1+\varepsilon_o].
\end{equation}

For any point $\bar x\in \Omega\backslash D$, there exists $\delta>0$ such that $B_\delta(\bar x)\subset \Omega\backslash D$.
Next we will construct a subsolution  in $B_\delta(\bar x)\times [0,2]$.
 Set
$$ \underline{u}(x,t)=\chi_D (x) \bar u (x,t)+\varepsilon\zeta(x)\eta(t) ,~(x,t)\in B_\delta(\bar x)\times [1-\varepsilon_o,1+\varepsilon_o],$$
where
\begin{equation*}
\chi_D (x)=
\left\{\aligned
&1,   & x  \in  D,\\
&0,    &  x \in   D^c,\\
\endaligned \right.
\end{equation*}
$\zeta(x)\in C^\infty_0(B_\delta(\bar x))$  satisfies
\begin{equation*}
\zeta(x)=\left\{\begin{array}{lll} 1,& x\in B_{\frac{\delta}{2}}(\bar x),\\
0,&  x\not \in B_{\delta}(\bar x),
\end{array}\right.
\end{equation*}
and $\eta(t)\in C^\infty_0([1-\varepsilon_o,1+\varepsilon_o])$ satisfies
\be \label{2020414}
  \eta(t)=\left\{\begin{array}{lll} 1,& t \in [1-\frac{\varepsilon_o}{2}, 1+ \frac{\varepsilon_o}{2}],\\
 0,&  t\not \in [1-\varepsilon_o,1+\varepsilon_o].
  \end{array}\right.\ee
For $(x,t)\in B_\delta(\bar x) \times [1-\varepsilon_o,1+\varepsilon_o]$, by \eqref{11}, we derive
$$\aligned &\frac{\partial \underline{u}}{\partial t}(x,t)+(-\lap)^s \underline{u}(x,t)\\
=&\varepsilon\eta'(t) \zeta(x) +(-\lap)^s \big(\chi_D\bar u(x,t)\big)+\varepsilon\eta(t)(-\lap)^s\zeta(x)\\
=& \varepsilon\eta'(t) \zeta(x)+\varepsilon\eta(t)(-\lap)^s\zeta(x)+C_{N,s}P.V.\int_{\mathbb R^N}\frac{-\chi_D\bar u(y,t)}{|x-y|^{N+2s}}dy\\
=&\varepsilon\eta'(t) \zeta(x)+\varepsilon\eta(t)C_1+C_{N,s}P.V.\int_{D}\frac{-\bar u(y,t)}{|x-y|^{N+2s}}dy\\
\leq&C_2\varepsilon+C_1\varepsilon-C.
\endaligned$$
By choosing $\varepsilon>0$ sufficiently small we have $$\frac{\partial \underline{u}}{\partial t}(x,t)+(-\lap)^s \underline{u}(x,t) \leq 0,\ \ ~(x,t)\in B_\delta(\bar x)\times [1-\varepsilon_o,1+\varepsilon_o]. $$
For $(x,t)\in B^c_\delta(\bar x) \times [1-\varepsilon_o,1+\varepsilon_o]$, we have
$$\underline{u}(x,t)=\chi_D \bar u (x,t)\leq \bar u(x,t)
.$$
As a result, set $v(x,t)=\bar u(x,t)-\underline{u}(x,t)$, combining with \eqref{12}, we have
$$\left\{\begin{array}{lll} \frac{\partial v}{\partial t}(x,t)+(-\lap)^sv(x,t) \geq 0,& (x,t)\in B_\delta(\bar x)\times [1-\varepsilon_o,1+\varepsilon_o],\\
v(x,t)\geq0,& (x,t)\in B^c_\delta(\bar x)\times [1-\varepsilon_o,1+\varepsilon_o],\\
v(x,0)\geq0,& x\in B_\delta(\bar x).
\end{array}\right.$$
Applying Lemma \ref{3} in the case of $c(x,t)\equiv0$, we obtain
 $$ v(x,t)\geq0,~(x,t)\in B_\delta(\bar x)\times [1-\varepsilon_o,1+\varepsilon_o].$$
Therefore $$\bar u(x,t) \geq\underline{u} (x,t)=\varepsilon\eta(t)  \zeta(x),~(x,t)\in B_\delta(\bar x)\times [1-\varepsilon_o,1+\varepsilon_o].$$
In particular, taking $t=1$, we obtain
$$e^{m}u_\infty(x,1)=e^{m}\varphi(x)\geq  \varepsilon \zeta(x),~x\in B_\delta(\bar x).$$
It follows that
\be\label{eq:4w399}
\varphi(\bar x)\geq e^{-m}\varepsilon>0.
\ee
By the arbitrariness of  $\bar x\in \Omega\backslash D$,  combining \eqref{eq:4w01} and \eqref{eq:4w399}  we obtain
$$\varphi (x) >0, ~x  \in \Omega.$$
\end{proof}

\begin{theorem}\label{thmj2a} (Asymptotic strong maximum principle for antisymmetric function)
For sufficiently large $\bar{t}$, suppose that $w_\lambda(x,t)\in (C^{1,1}_{loc}(\Sigma_{\lambda})\cap {\mathcal L}_{2s} )\times C^1([\bar{t},\infty)) $ is bounded and   satisfies
\begin{equation}\label{meq1a}
\left\{\aligned
&\frac{\partial w_\lambda}{\partial t}+(-\Delta)^sw_\lambda=c_\la(x,t)w_\lambda,  & (x,t)\in \Sigma_{\lambda} \times [\bar{t},\infty),\\
&w_\lambda(x,t)=-w_\lambda(x^\lambda,t),  &(x,t)\in  \Sigma_{\lambda} \times [\bar{t},\infty),\\
&\underset{t \rightarrow \infty}{\underline{\lim}} w_\lambda(x,t)\geq 0, & x\in \Sigma_\la,
\endaligned \right.
\end{equation}
with $c_\la(x,t)$ is bounded.
Assume $\psi_\lambda >0$ somewhere in $\Sigma_\lambda$.

Then $\psi_\lambda(x)>0$ in $\Sigma_{\lambda}$.
\end{theorem}

In order to prove Theorem \ref{thmj2a}, we first derive
\begin{lemma}\label{thm4w98}(Asymptotic maximum principle for antisymmetric functions)
Let $\Omega$ be a bounded domain in $\Sigma_\la$. Assume that   $v(x,t)\in (C^{1,1}_{loc}(\Omega)\cap {\mathcal L}_{2s})\times C^1([0,\infty))$ is lower semi-continuous in $x$ on $\bar{\Omega}$ and satisfies
\begin{equation*}
\left\{
\begin{array}{ll}
 \frac{\partial v}{\partial t}(x,t)+(-\lap)^s v(x,t)\geq
 c_\la(x,t)v(x,t), & (x,t)\in \Omega \times [0,\infty),\\
 v(x^\la,t)=-v(x,t), &  (x,t)\in  \Sigma_\la  \times [0,\infty),\\
 v(x,t)\geq 0, & (x,t)\in ( \Sigma_\la \backslash \Omega )\times [0,\infty),\\
 v(x,0) \geq 0, &  x \in  \Omega  .
 \end{array}
\right.
\end{equation*}
If $c_\la(x,t) $ is bounded from above, then
$$v(x,t)\geq 0,~(x,t)\in \Omega\times[0,T],\ \ \forall\ T>0.$$
\end{lemma}
One can derive Lemma \ref{thm4w98} by a similar idea as in the proof of Lemma \ref{3}, here we omit the details.
%
%
%

\begin{proof}[\bf Proof of Theorem \ref{thmj2a}]

For any  $\varphi\in \omega(u)$, there  exists $t_k$ such that
$w_\lambda(x,t_k)\rightarrow \psi_\lambda(x)$ as $t_k\rightarrow \infty$.
Set $$w_k(x,t)=w_\lambda(x,t+t_k-1).$$
Then
$$\frac{\partial w_k}{\partial t}(x,t)+(-\Delta)^sw_ k(x,t)=c_k(x,t)w_ k(x,t),\ \ (x,t)\in \Sigma_{\lambda} \times [\bar{t},\infty),$$
where $c_k(x,t)=c_\lambda(x,t+t_k-1).$
By standard parabolic regularity estimates \cite{XFR}, we infer that there is a subsequence of $w_ k(x,t)$ (still denoted by $w_ k(x,t)$) which converges uniformly to a function $w_ \infty(x,t)$ in $\Sigma_\lambda\times[0,2]$ and
\begin{align*}
&\frac{\partial w_ k}{\partial t}(x,t)+(-\Delta)^sw_ k(x,t)\rightarrow \frac{\partial w_ \infty}{\partial t}(x,t)+(-\Delta)^sw_ \infty(x,t),\\
&c_k(x,t)\rightarrow c_\infty(x,t),\ \ \ \ \mbox{as}\ k\rightarrow\infty.
\end{align*}
Also, since $c_\lambda$ is bounded, so does $c_\infty$, we can deduce from \cite{XFR} that $w_ \infty(x,t)$ is H\"older continuous in $x$ and $t$.
Particularly,
$$w_\lambda(x,t_k)=w_k(x,1)\rightarrow w_\infty(x,1)=\psi_\lambda(x)\, \ \mbox{as}\ k\rightarrow \infty.$$
Take $m>0$ such that
$$
c_\infty(x,t)+m>0.
$$
Consider a new function $$\tilde{w}(x,t)=e^{mt}w_ \infty(x,t).$$ By the third condition in \eqref{meq1a},  we have
$$\tilde{w}(x,t)\geq0,\ \ \text{ in}\ \Sigma_\lambda\times[0,2]$$
and thus
\begin{equation}\label{mthm1}
\frac{\partial \tilde{w}}{\partial t}(x,t)+(-\Delta)^s\tilde{w}(x,t)=(c_\infty(x,t)+m)\tilde{w}(x,t)\geq0,~(x,t)\in \Sigma_\lambda\times[0,2].
\ee

Since $\psi_\lambda>0$ somewhere in $\Sigma_\lambda$, by continuity, there exists a set $D\subset \subset \Sigma_{\lambda}$ such that
\begin{equation}\label{meq01}
\psi_\lambda(x)>c>0,\ \, \  x \in D,
\end{equation}
with positive constant $c$.

By the continuity of $w_\infty(x,t)$, there exist $0<\varepsilon_o<1$, such that
$$ w_\infty (x,t) > c/2, \;\; (x,t) \in D \times [1-\varepsilon_o,1+\varepsilon_o].$$
For simplicity of notation, we may assume that
\be w_\infty (x,t) > c/2, \;\; (x,t) \in D \times [0,2].
\label{A50}
\ee

For any point $\bar x\in \Sigma_\la \backslash D$, choose $\delta=\min \{dist(\bar{x}, D), dist(\bar{x}, T_\lambda)\}>0$, then $B_\delta(\bar{x})\subset \Sigma_\la \backslash D$.

\begin{center}
\begin{tikzpicture}[node distance = 0.3cm]
\draw [->, semithick] (-4,0) -- (4,0) node[right] {$x_1$};
\draw [ semithick] (0,-1) -- (0,3.8) node[above] {$T_\lambda$};
\path (-2.3,3.7) node [ font=\fontsize{10}{10}\selectfont] {$\Sigma_\lambda=\{ x\in \mathbb R^N\mid x_1<\lambda\}$};
\draw (-0.6,2.8) circle (0.6);
\path (-0.6,2.8) [very thin,fill=black]  circle(1 pt) node at (-0.6, 2.5) {$\bar x$};
\node (test) at (2,2.7) {$B_{\delta}(\bar x)$};
    \draw [->,thin] (-0.1,2.8) to [in = 90, out = 60] (2,3) (test);
\draw (-2,1) ellipse  [x radius=1.5cm, y radius=0.8cm];
\fill [orange] (-2,1) ellipse  [x radius=1.5cm, y radius=0.8cm];
\draw (2,1) ellipse  [x radius=1.5cm, y radius=0.8cm];
\fill [orange] (2,1) ellipse  [x radius=1.5cm, y radius=0.8cm];
\path (-2.5,1) node [ font=\fontsize{10}{10}\selectfont] {$D$};
\path (2.5,1) node [ font=\fontsize{10}{10}\selectfont] {$D_\lambda$};
\node [below=1cm, align=flush center,text width=8cm] at (0,-0.1)
        {$Figure$ 1 };
\end{tikzpicture}
\end{center}
Next, we construct a subsolution in $B_\delta(\bar{x})\times[0,2].$
Set
 $$\underline{w}(x,t)=\chi_{D\cup D_\lambda} (x) \tilde{w}(x,t)+\varepsilon \eta(t)g(x),$$ where $D_\lambda$ is the reflection of $D$ about the plane $T_\lambda$, $\eta(t)\in C^\infty((0,2))$  be as defined  in \eqref{2020414}, $g(x)=(\delta^2-|x-\bar{x}|^2)_+^{s}-(\delta^2-|x-\bar{x}^\lambda|^2)_+^{s}$. Obviously,  $g(\bar x)=\delta^{2s}$, $g(x^\lambda)=-g(x).$ We also have
 \be\label{56}
 (-\Delta)^sg(x)\leq C_0,
 \ee
 where $C_0$ is a constant, see Lemma \ref{mlem1} in the Appendix for the proof.

By the definition of fractional Laplacian and \eqref{A50}, we derive that for each fixed $t \in [0,2]$ and for any $x\in B_\delta(\bar{x})$
\begin{align}\label{mthm11}
(-\Delta)^s(\chi_{D\cup D_\lambda}(x)\tilde{w}(x,t))
&= C_{N,s} P.V. \int_{\mathbb R^N}\frac{\chi_{D\cup D_\lambda}(x)\tilde{w}(x,t)-\chi_{D\cup D_\lambda}(y)\tilde{w}(y,t)}{|x-y|^{N+2s}}\;dy\nonumber\\
&= C_{N,s} P.V. \int_{\mathbb R^N}\frac{-\chi_{D\cup D_\lambda}\tilde{w}(y,t)}{|x-y|^{N+2s}}\;dy\nonumber\\
&= C_{N,s} P.V. \int_{D}\frac{-\tilde{w}(y,t)}{|x-y|^{N+2s}}\;dy+\int_{D}\frac{-\tilde{w}(y^\lambda,t)}{|x-y^\lambda|^{N+2s}}\;dy\nonumber\\
&= C_{N,s} P.V. \int_{D}\left(\frac{1}{|x-y^\lambda|^{N+2s}}-\frac{1}{|x-y|^{N+2s}}\right)\tilde{w}(y,t)\;dy\nonumber\\
&\leq -C_1,
\end{align}
where $C_1$ is a positive constant.
By \eqref{mthm11} and \eqref{56}, we derive
\begin{align*}
\frac{\partial \underline{w}(x,t)}{\partial t}+(-\Delta)^s\underline{w}(x,t)&=\varepsilon \eta'(t)g(x)+(-\Delta)^s(\chi_{D\cup D_\lambda}\tilde{w}(x,t))+\varepsilon \eta(t)(-\Delta)^sg(x)\nonumber\\
&\leq \varepsilon \eta'(t)g(x)-C_1+\varepsilon \eta(t) C_0,~(x,t)\in B_\delta(\bar{x})\times[0,2].
\end{align*}
Choose $\varepsilon>0$ sufficiently small such that
\begin{equation}\label{mthm2}
\frac{\partial \underline{w}}{\partial t}(x,t)+(-\Delta)^s\underline{w}(x,t)\leq0,\ \,\ (x,t) \in B_\delta(\bar{x})\times[0,2].
\end{equation}

Set $v(x,t)=\tilde{w}(x,t)-\underline{w}(x,t)$. Obviously, $v(x,t)=-v(x^\lambda,t)$. It follows from \eqref{mthm1} and \eqref{mthm2} that
$v(x,t)$ satisfies
\begin{equation*}
\frac{\partial v}{\partial t}(x,t)+(-\Delta)^sv(x,t)\geq0,~(x,t) \in B_\delta(\bar{x})\times[0,2].
\end{equation*}
Also, by the definition of $\underline{w}(x,t)$, we have
$$v(x,t)\geq 0\ \ \text{in}\ (\Sigma_\lambda\setminus B_\delta(\bar{x}))\times[0,2]$$
and $$
v(x,0)\geq 0\ \ \text{in}\ \Sigma_\lambda.
$$
Now, by applying Lemma \ref{thm4w98},  we obtain $$
v(x,t)\geq0,\ \ (x,t) \in  B_\delta(\bar{x})\times[0,2].
$$
It implies that
$$v(x,t)=e^{mt}w_\infty(x,t)-\varepsilon g(x)\eta(t)\geq0, \ \,\ (x,t) \in  B_\delta(\bar{x})\times[0,2].$$
In particular,
$$w_\infty(x,1)\geq e^{-m}\varepsilon g(x) ,\ \, \  x \in  B_\delta(\bar{x}).$$
Since $g(\bar x)= \delta^{2s}$, we conclude
\be \label{eq:ww2999}
\psi_\lambda(\bar x) = w_\infty(\bar x,1) \geq e^{-m}\varepsilon \delta^{2s}>0.
\ee
By the arbitrariness of  $\bar x$ in $\Sigma_\lambda \backslash D$,  combining \eqref{meq01} with \eqref{eq:ww2999}, we obtain
$$\psi_\lambda (x) >0, ~x  \in \Sigma_\lambda.$$

This complete the proof of Theorem \ref{thmj2a}.
\end{proof}

\section{Asymptotic symmetry of solutions in \texorpdfstring{$B_1(0)$}{B}}

In this section, we prove the asymptotic symmetry of solutions for problem \begin{equation}\label{eq:jj20a}
\left\{\begin{array}{lll}
\frac{\partial u}{\partial t}(x,t)+(-\lap)^s u(x,t) =f(t,u(x,t)), &\qquad  (x,t)\in B_1(0) \times (0,\infty),\\
u(x,t) =0, &\qquad (x,t)\in B^c_1(0) \times (0,\infty).
\end{array}\right.
\end{equation}

\begin{theorem}\label{thmj5a}
Assume that $u(x,t)\in \Big(C^{1,1}_{loc} \big(B_1(0)\big) \cap C\big(\overline{B_1(0)}\big)\Big)\times C^1\big((0,\infty)\big)$ is a positive uniformly bounded solution of \eqref{eq:jj20a}.

Let $\alpha\in(0,1)$  such that $\frac{\alpha}{2s}\in(0,1)$. Suppose $f(t,u)\in L^\infty(\mathbb{R}^+\times\mathbb{R})$  is $ C^{\frac{\alpha}{2s}}_{loc}$  in $t$,  Lipschitz continuous in $u$ uniformly with respect to $t$ and
\begin{equation}
\label{46a}
f(t,0)\geq0,\ \ \ t\geq0.
\end{equation}

Then for all $\varphi(x) \in \omega(u)$, either  $\varphi(x)\equiv0$ or $\varphi(x)$ are radially symmetric and decreasing about the origin.
\end{theorem}
\begin{remark}
By Theorem \ref{thmwj2a}, if we assume $f$ satisfies more
stronger conditions $(F)$ in Theorem \ref{thmj5a}, we can obtain stronger result, that is either all $ \varphi (x) \in \omega (u)$ are identically 0 or all $ \varphi (x) \in \omega (u)$ are radially symmetric and decreasing about some point in $\mathbb R^N$.
\end{remark}
Let $T_\lambda, \Sigma_\lambda,  x^\lambda, u_\lambda, w_\lambda, \varphi$  and $\psi_\lambda$ be defined as  in Section 1.
Set $$\Omega_\lambda= \Sigma_\lambda \cap B_1(0)=\{x\in B_1(0) \mid x_1< \lambda\},$$
then $w_\lambda$ satisfies
\begin{equation}\label{13}
\left\{
\begin{array}{ll}
   \frac{\partial w_\la}{\partial t}+(-\lap)^s w_{\lambda}   = c_\lambda (x,t) w_\lambda  ,~&(x,t)\in \Omega_\lambda\times (0,\infty), \\
  w_\lambda(x,t)=- w_\lambda(x^\lambda,t), ~&(x,t)     \in \Omega_\lambda\times(0,\infty) ,
\end{array}
\right.
\end{equation}
where $c_\lambda (x,t)=\frac{f(t,u_\lambda(x,t))-f(t,u(x,t))}{u_\lambda(x,t)-u(x,t)}$.

We will carry on the proof in two steps. Choose any ray from the origin as the positive $x_1$ direction. First we show that for $\lambda>-1$ but sufficiently close to $-1$, we have for all $\varphi \in \omega(u)$
\begin{equation}\label{eq:jj21}
\psi_{\lambda}(x) \geq0, \;\; \forall \, x \in \Omega_\lambda.
\end{equation}
This provides the starting point to move the plane. Then we move the plane $T_{\lambda}$ toward the right as long as inequality (\ref{eq:jj21}) holds to its limiting position. Define
\begin{equation}\label{eq:jj22}
\lambda_0 = \sup \{\lambda \leq0 \mid  \psi_\mu(x)\geq 0, ~\mbox{for~ all } ~\varphi \in \omega(u),~ x \in \Omega_{\mu},~ \mu \leq \lambda \}.
\end{equation}
We will show that $\lambda_0 =0$. Since $x_1$ direction can be chosen  arbitrarily, we deduce that for any $\varphi \in \omega(u)$,   $\varphi(x)$ is radially symmetric and monotone decreasing about the origin. We now show the details in the two steps.
\begin{proof}[\bf Proof of Theorem \ref{thmj5a}]
For any  $ \varphi \in \omega(u)$, if $\varphi(x) \equiv 0,$ the conclusion of Theorem \ref{thmj5} is trivial. Without loss of generosity, we assume that for any $ \varphi \in \omega(u)$, $\varphi \not\equiv 0\ \text{in } B_1(0)$.

{\em Step 1.} We show that for $\lambda> -1$ and sufficiently close to $-1$,
\begin{equation}\label{eq:jjw21}
\psi_{\lambda}(x)\geq0, \;\; \forall \, x \in \Omega_\lambda,~\mbox{for~all}~\varphi \in \omega(u).
\end{equation}

The Lipschitz continuity assumption of $f$ guarantees that $c_\lambda(x,t)$ is  bounded. We also have  $$
  w_\lambda(x,t)\geq 0 , ~x      \in \Sigma_\lambda  \backslash \Omega_\la ,~t \in (0,\infty)
 $$
  since $u(x,t)=0,~(x,t)\in B^c_1(0) \times (0,\infty).$
Combining with \eqref{13}, we can apply Theorem \ref{thmj1} $(\textbf{i})$ to conclude that \eqref{eq:jjw21} holds.

{\em Step 2.}
We will show that $\lambda_0$ defined in \eqref{eq:jj22} satisfies
\begin{equation*}
\lambda_0  = 0 .
\end{equation*}
 Suppose that $\lambda_0  < 0$. We will prove that  $T_{\lambda_0}$ can be moved further to the right a little bit, and this is a contradiction with the definition of $\lambda_0$.

By \eqref{eq:jj22}, we have $$
\psi_{\lambda_0 }(x)\geq 0 ~\mbox{for~ all } ~\varphi \in \omega(u),~ x \in \Omega_{\lambda_0 }.
$$
First we prove that for any $\varphi \in \omega(u)$, there exists $x_\varphi\in\Sigma_{\lambda_0 }$ such that
\be\label{48}
\psi_{\lambda_0 }(x_\varphi)>0.
\ee
If not, there exists some $\bar{\varphi }\in \omega(u)$ such that $$
\bar{\psi}_{\lambda_0 }(x)=\bar{\varphi }_{\lambda_0 }(x)-\bar{\varphi }(x)\equiv0\ \ \text{in}\ \Sigma_{\lambda_0 }.
$$
Also, by the outer condition of $u$, we have $\bar{\varphi }(x)\equiv0$ in $B_1^c(0)\cap\Sigma_{\lambda_0 }$, therefore, there exists $x_0\in B_1(0)$ such that $\bar{\varphi }(x_0)=0$.

For this $\bar{\varphi },$ there exists $t_k$ such that $u(x,t_k)\rightarrow\bar{\varphi }(x)$ as $t_k\rightarrow \infty$.
 Then, by a similar translation and regularity process to \eqref{49}, we have
$$
 \frac{\partial u_\infty}{\partial t}(x,t)+(-\lap)^su_\infty(x,t)= \tilde{f}(t,u_\infty(x,t)),
$$
and $u_\infty(x,1)=\bar{\varphi }(x)$. Noting that $u_\infty(x,t)\geq0$, we obtain  $\frac{\partial u_\infty}{\partial t}(x_0,1)\leq0$ and
$$
(-\lap)^su_\infty(x_0,1)=C_{N,s}P.V.\int_{B_1(0)}\frac{- u_\infty(y,1)}{|x_0-y|^{N+2s}}dy
<0,
$$
where we used in the last inequality that $u_\infty(y,1)\not\equiv0$ in $B_1(0)$ $($since $\bar{\varphi }\not\equiv0$ in $B_1(0))$.

As a result, $\tilde{f}(1,u_\infty(x_0,1))=\tilde{f}(1,0)<0,$ which contradicts \eqref{46a}. So, we obtain \eqref{48}.

Now by the {\em  asymptotic strong maximum principle for antisymmetric functions} (Theorem \ref{thmj2}), we have for all $\varphi\in \omega(u)$
\be\label{16}\psi_{\lambda_0 }(x)>0,~x\in \Omega_{\lambda_0 }.\ee
Thus for any $\delta>0$ small, for each $\psi_{\la_0 }\big(\text{corresponding to }\varphi\in\omega(u)\big)$, there exists a constant $C_{\varphi}>0$  (depending on $\varphi$), such that
\be \label{eq:4w2222}
 \psi_{\lambda_0 }(x)\geq C_{\varphi}>0, ~x\in \overline{\Omega_{\lambda_0 -\delta}}.
 \ee

We claim further that, for all $\varphi \in \omega(u)$, there exists a universal constant $C_0$ such that
\begin{equation}\label{eq:jj51}
\psi_{\lambda_0 }(x)\geq C_0>0,~x\in \overline{\Omega_{\lambda_0 -\delta}}.
\end{equation}
Otherwise, there exists a sequence of functions $\{\psi_{\lambda_0}^k\} \big(\text{corresponding to }\{\varphi^k\} \subset  \omega(u)\big)$ and a sequence of points $\{x^k\}\subset \overline{\Omega_{\lambda_0-\delta}}$ such that
\be\label{15} \psi_{\lambda_0}^k(x^k)< \frac{1}{k}.\ee
Due to the compactness of $ \omega(u)$ in $C_0( \overline{B_1(0)})$,  there exists $\psi^0_{\lambda_0}\big(\text{corresponding to some }\ \varphi^0 \in \omega(u)\big)$ and $x_0\in\overline{\Omega_{\lambda_0-\delta}}$ such that

$$\psi_{\la_0}^k(x^k)\rightarrow \psi^0_{\la_0}(x^0)~\ \ \mbox{as}~k\rightarrow \infty.$$
Now due to \eqref{15}, we obtain
$$\psi_{\lambda_0}^0(x^0)=0.$$
This contradicts \eqref{eq:4w2222}, since $\varphi^0 \in \omega(u)$. Hence \eqref{eq:jj51} must be true.

From \eqref{eq:jj51} and the continuity of $\psi_\lambda$ with respect to $\lambda$, for each $\psi_\lambda\big(\text{corresponding to }\varphi\in \omega(u)\big)$, there exist $\varepsilon_{\varphi}(\text{depending on}\  \varphi) >0$ such that
$$  \psi_\lambda (x) \geq\frac{C_0}{2}>0,~x\in \overline{\Omega_{\lambda_0-\delta}},\ \ \ \ ~\forall \ \lambda \in (\lambda_0,\lambda_0+\varepsilon_{\varphi}).$$

Through a similar compactness argument as in deriving \eqref{eq:jj51}, we conclude that, for all $\psi_\lambda$, there exists a universal $\varepsilon >0$ such that
\begin{equation}\label{eq:jj26}
\psi_\lambda (x) \geq \frac{C_0}{2}>0,~x\in \overline{\Omega_{\lambda_0-\delta}},\ \ \ \ ~\forall \ \lambda \in (\lambda_0,\lambda_0+\varepsilon).
\end{equation}

As a consequence, for $t$ sufficiently large, we have
$$w_\lambda(x,t)\geq0,\ \ ~x\in \overline{\Omega_{\lambda_0-\delta}},\ \ \ \ ~\forall \ \lambda \in (\lambda_0,\lambda_0+\varepsilon).$$

Since $\delta>0$ is small, also, by \eqref{eq:jj26}, we can choose $\varepsilon>0$ small, such that $\Omega_\lambda \backslash \Omega_{\lambda_0-\delta}$  is a narrow region for $\lambda \in (\lambda_0,\lambda_0+\varepsilon)$, thus applying
{\em the  asymptotic narrow region principle} (Theorem \ref{thmj1}), we arrive at
\begin{equation}\label{eq:jj27}
\psi_\lambda(x) \geq 0,~ \forall ~x\in \Omega_\lambda \backslash \Omega_{\lambda_0-\delta}.
\end{equation}
Combining \eqref{eq:jj26} and \eqref{eq:jj27}, we conclude that
$$\psi_\lambda(x) \geq 0,~\forall~x\in \Omega_\lambda,~\forall\ \lambda\in (\lambda_0,\lambda_0+\varepsilon),\ \ \ \forall\varphi\in \omega(u).$$
This contradicts the definition of $\lambda_0$. Therefore, we must have $\lambda_0 = 0$.
It follows that for   all $\varphi\in \omega(u)$
$$\psi_0(x) \geq 0,~\forall~x\in \Omega_0,$$
or equivalently, for   all $\varphi\in \omega(u)$
\begin{equation}\label{eq:jj28}
\varphi(-x_1,\cdots,x_N)\leq \varphi(x_1,\cdots,x_N),~0<x_1<1.
\end{equation}
Since the $x_1$-direction can be chosen arbitrarily, \eqref{eq:jj28} implies that all $\varphi(x) $ are radially
symmetric about the origin.

The monotonicity can be deduced from $$ \psi_\lambda(x)>0,~x\in \Omega_\lambda,\ \forall \ -1<\lambda <0.$$
which can be derived through a similar process as in the proof for \eqref{16},
Now we complete the proof of Theorem \ref{thmj5a}.
\end{proof}

\section{Asymptotic symmetry of solutions in \texorpdfstring{$\mathbb R^N$}{R}}
\label{s:prescribing}
In this section, we will use an asymptotic method of moving planes to prove asymptotic symmetry of solutions to problem
\begin{equation}\label{eq:j1a}
\frac{\partial u}{\partial t}+(-\lap)^s u=f(t,u),~(x,t) \in \mathbb R^N \times (0,\infty).
\end{equation}

\begin{theorem} \label{thmj6a}
Assume $f$ satisfies $(F)$. Let $u(x,t) \in \big( C^{1,1}_{loc}(\mathbb R^N)\cap {\mathcal L}_{2s} \big)\times C^1\big((0,\infty)\big)$ be a positive uniformly bounded solution of \eqref{eq:j1} satisfies
 \begin{equation}\label{eq:jj200a}
\underset{|x|\rightarrow \infty }{\lim}~ u(x,t) =0,\ \ \text{uniformly for sufficiently large }t.
\end{equation}

Then we have

either all $ \varphi (x) \in \omega (u)$ are identically 0 or

all $ \varphi (x) \in \omega (u)$ are radially symmetric and decreasing about some point in $\mathbb R^N$. That is,  there exists $\tilde x \in \mathbb R^N$ such that $$ \varphi(x-\tilde x)=\varphi(|x-\tilde x|),~x\in \mathbb R^N.$$
\end{theorem}

Let $T_\lambda, \Sigma_\lambda,  x^\lambda, u_\lambda, w_\lambda, \varphi$  and $\psi_\lambda$ be defined as in Section 1.
Then from \eqref{eq:j1a}, we have $w_\lambda(x,t)$ satisfies
\begin{equation}\label{17}
\left\{
\begin{array}{ll}
   \frac{\partial w_\la}{\partial t}+(-\lap)^s w_{\lambda}= c_\lambda (x,t) w_\lambda  ,~& (x,t)\in \Sigma_\lambda\times (0,\infty),\\
  w_\lambda(x,t)=- w_\lambda(x^\lambda,t), ~&(x,t)     \in \Sigma_\lambda\times(0,\infty) ,
\end{array}
\right.
\end{equation}
where $$c_\lambda (x,t)=f_u(t, \xi_\lambda(x,t)),\ \ \xi_\lambda(x,t)\ \text{is valued between}\ u_\lambda(x,t)\ \text{and}\ u(x,t).$$

\begin{proof}[\bf Proof of Theorem \ref{thmj6a}]
If all $\varphi\in \omega(u)$ are identically $0$ not holds,   without loss of generality,  we assme that there is some $\varphi\in \omega(u)$ which is positive somewhere in $\mathbb R^N$.
Then, we can deduce from Theorem \ref{thmwj2} that for all $\varphi\in \omega(u)$,
\be\label{44} \varphi(x)>0,~x\in \mathbb R^N.\ee
We divide the proof into three steps.
\subsection{Step 1: \texorpdfstring{$\lambda$}{BB} sufficiently negative.}

We prove that for $\lambda$ sufficiently negative,
\be\label{43}
\psi_\lambda(x)>0\ \ \  \text{in}\ \Sigma_\lambda,\ \ \forall\ \varphi\in \omega(u).\ee
\medskip
We first show that,
\be\label{23}
\psi_\lambda(x)\geq0\ \ \  \text{in}\ \Sigma_\lambda,\ \ \forall\ \varphi\in \omega(u).\ee

Since $f_u(t,0)<-\sigma$ and $f_u$ is continuous near $u=0$, then there exists a $\beta>0$ such that for any
$$
0 \leq \eta <\beta,
$$
we have \be\label{19}
f_u(t,\eta)<-\sigma.
\ee
 For this $\beta>0$, by \eqref{eq:jj200a}, there exists $R>0$ large, such that for any $|x|>R$, we have
\be\label{20}
0< u(x,t)<\beta,\ \ \text{uniformly for sufficiently large }t.
\ee

At the points where $w_\lambda(x,t)<0$, we have
\be\label{22}
u_\lambda(x,t)\leqslant\xi(x,t)\leqslant u(x,t).\ee

Combing
\eqref{19}, \eqref{20} and \eqref{22}, we obtain for sufficiently large $t$,
\be\label{21}
c_\lambda (x,t)=f_u
(t,\xi_\lambda(x,t))<-\sigma,\ \ \ \forall\  |x|>R,\ \ \lambda\in \mathbb{R}.
\ee
Also, by \eqref{eq:jj200a}, we have
 \be\label{24}
 \underset{|x|\rightarrow \infty}{\underline{\lim}}~w_\lambda(x,t)= 0,\ \ \lambda\in \mathbb{R},\ \ \text{uniformly for sufficiently large }t.
 \ee
Hence, for $\lambda\leqslant-R$, combining \eqref{17},   \eqref{21} and \eqref{24}, applying Theorem \ref{1} with
$\Omega=\Sigma_\lambda$,  we obtain \eqref{23}.

To apply the strong maximum principle to derive (\ref{43}), we only need to verify that $\psi_\lambda$ is positive somewhere.

Actually, by \eqref{44}, for any $r>0 \; (r<R)$, $\varphi\in \omega(u)$, there exists a constant $C_{r,\varphi}>0$ such that
$$
\varphi(x)\geqslant C_{r,\varphi}>0,~x\in B_r(0).$$
From \eqref{eq:jj200a}, for this $C_{r,\varphi}>0$, there exists some $\lambda_\varphi\leqslant-R$ satisfying
$$
\varphi(x)\leqslant \frac{C_{r,\varphi}}{2},~x\in B_r(0^{\lambda_\varphi}).
$$
Due to the compactness of $\omega(u) $ in $C_0(\mathbb{R}^N)$, the above $C_{r,\varphi}$ and $\lambda_\varphi$ can be chosen uniformly with respect to all $\varphi$ in $\omega(u)$, denote them by $C_r$ and $\lambda$, respectively.
Then
\be\label{45}
\psi_\lambda(x)\geqslant\frac{C_r}{2}>0,\ \ x\in B_r(0^\lambda),\ \ \ \forall\ \varphi\in \omega(u).
\ee
Combining \eqref{23}, \eqref{45}  and the  boundedness of $c_\lambda(x,t)$,  we deduce \eqref{43} by Theorem \ref{thmj2}.

\subsection{Step 2: Move the plane to the rightmost limiting position.}

Inequality \eqref{23} provides a starting point, from which we move the plane $T_\lambda$ toward the right as long as (\ref{23}) holds to its limiting position. More precisely, let
$$ \lambda_0^- = \sup \{ \lambda \mid    \psi_\mu (x)  >0, ~ \forall \varphi\in \omega(u), \; x \in \Sigma_\mu, \; ~ \mu \leq \lambda \}. $$

We prove that

$(i)$ there is at least some  $\varphi(x) \in \omega(u)$ which is symmetric about the limiting plane $T_{\lambda_0^-}$, that is
\begin{equation}\label{eq:j110}
 \psi_{\lambda_0^-}(x) \equiv 0 , \;\; x \in \Sigma_{\lambda_0^-},~\mbox{for~some}~\varphi(x)\in \omega(u);
\end{equation}
and

$(ii)$ $\partial_{x_1} \varphi (x) > 0,~ x \in \Sigma_{\la_0^-},\ \ \mbox{ for all } \varphi \in \omega(u).$
\medskip

Suppose that \eqref{eq:j110} is false, we will prove that there exists $\varepsilon_0>0$ such that for all $\varphi \in \omega(u)$
\be \label{eq:4w00}
\psi_\la(x) > 0,~ x \in \Sigma_\la,~\ \forall\ \la\in(\la_0^-,\la_0^-+\varepsilon_0).
\ee

In fact, if \eqref{eq:j110} does not hold, then for every $\varphi \in \omega(u)$,  there exists $x_\varphi\in \Sigma_{\lambda_0^-}$ such that $\psi_{\lambda_0^-}(x_\varphi)>0$. Also, observing that $\psi_{\lambda_0^-}(x)\geq0$ in $\Sigma_{\lambda_0^-}$,  by the  asymptotic strong maximum principle for antisymmetric functions (Theorem \ref{thmj2}), we have
\be\label{eq:j420}
\psi_{\lambda_0^-} (x) > 0 , \;\; \forall \, x \in \Sigma_{\lambda_0^-},\ \ \forall \ \varphi \in \omega(u).
\ee


Let $R$ be the same as in \eqref{21}.

We first consider the case $x\in  \overline{\Sigma_{\la_0^--\delta}\cap B_R(0)}$ for any given small $\delta>0$. By \eqref{eq:j420}, similar to the proof of \eqref{eq:jj26},
We derive that there exists $C_0>0$
and $\varepsilon_0>0$ such that
\begin{equation}\label{eq:j111}
\psi_{\lambda}(x)\geq\frac{C_0}{2}> 0,~ x\in \overline{ \Sigma_{\la_0^--\delta}\cap B_R(0)},~\la\in(\la_0^-,\la_0^-+\varepsilon_{0}),\ \ \forall \ \varphi \in \omega(u),
\end{equation}
which implies that for $t$ large
\begin{equation*}
w_{\lambda}(x,t)\geq 0,~ x\in \overline{ \Sigma_{\la_0^--\delta}\cap B_R(0)},~\la\in(\la_0^-,\la_0^-+\varepsilon_{0}),\ \ \forall \ \varphi \in \omega(u).
\end{equation*}
Now by \eqref{21}, \eqref{24} and the boundedness of $c_\lambda(x,t)$,  applying Theorem \ref{41a}, we have

\be\label{27}
\psi_{\lambda}(x)\geq 0,~ x\in\Sigma_{\la}\Big\backslash\overline{ \Sigma_{\la_0^--\delta}\cap B_R(0)},~\la\in(\la_0^-,\la_0^-+\varepsilon_{0}),\ \ \forall \ \varphi \in \omega(u).
\ee
Above all, combining \eqref{eq:j111} and \eqref{27}, we obtain
for all $\varphi \in \omega(u)$,
$$\psi_\la(x) \geq 0,~ x \in \Sigma_\la,~\ \forall\ \la\in(\la_0^-,\la_0^-+\epsilon_0).
$$
Now \eqref{eq:j111} enable us to employ the {\em  asymptotic strong maximum principle for anti-symmetric functions} (Theorem \ref{thmj2}) to arrive at \eqref{eq:4w00}, which contradicts the definition of $\la_0^-$. This verifies \eqref{eq:j110}.

$(ii)$ We prove that for each $\varphi \in \omega(u)$, it holds
\be \label{weq:4w4444}
\partial_{x_1} \varphi (x) > 0,~ x \in \Sigma_{\la_0^-}.
\ee

Indeed, fix any $\la<\la_0^-$, by the definition of $\la_0^-$, we have for any $\varphi \in \omega(u)$,
$$\psi_\la(x)>0.\ \ \forall x\in \Sigma_\lambda.$$
Due to the bounded-ness of $c_\lambda(x,t)=f_u
(t,\xi_\lambda(x,t))$, we are now able to apply

\begin{lemma}\label{lem4wa} (Asymptotic Hopf lemma for antisymmetric functions \cite{CWp})
Assume that $$ w_\lambda (x,t) \in(C^{1,1}_{loc}(\Sigma_{\lambda})\cap {\mathcal L}_{2s} )\times C^1((0,\infty))$$is bounded and satisfies
\begin{equation*}
\left\{\aligned
&\frac{\partial w_\lambda }{\partial t}+(-\Delta)^sw_\lambda=c_\lambda(x,t)w_\lambda,  &  (x,t) \in  \Sigma_{\lambda} \times (0,\infty),\\
&w_\lambda(x^\lambda,t)=-w_\lambda(x,t),  & (x,t) \in   \Sigma_{\lambda} \times (0,\infty),\\
&\underset{t \rightarrow \infty}{\underline{\lim}} w_\lambda(x,t)\geq 0 , &  x  \in  \Sigma_{\lambda} ,
\endaligned \right.
\end{equation*}
where
$$\underset{x\rightarrow \partial \Sigma_\la}{\underline{\lim}}c_\lambda(x,t)=o(\frac{1}{[dist(x,\partial \Sigma_\lambda)]^2}), \ \ \text{uniformly for sufficiently large $t$}.$$
If $\psi_\lambda>0$ somewhere in $\Sigma_\la$. Then $$ \frac{\partial \psi_\lambda}{\partial \nu}(x)<0,~x\in \partial \Sigma_\lambda,$$ where $\nu$ is an outward normal vector.
\end{lemma}

As a result of this lemma,
$$
\partial_{x_1} \psi_\la(x)\big|_{x\in T_\lambda}<0,\ \ \forall \ \varphi \in \omega(u).
$$
Noticing that
$$
\partial_{x_1} \psi_\la(x)\big|_{x\in T_\lambda}=-2 \partial_{x_1} \varphi (x) \Big|_{x\in T_\lambda},
$$
we have
$$\partial_{x_1} \varphi (x) \Big|_{x\in T_\lambda}>0,\ \ \forall \ \varphi \in \omega(u).$$
By the arbitrariness of $\la<\la_0^-$, we obtain \eqref{weq:4w4444}.

\subsection{Step 3: \texorpdfstring{All $\omega$-limit functions are radially symmetric.}{BB1}  }

We will prove that $\psi_{\la_0^-}(x)\equiv 0$ for all $\varphi(x)$.
\medskip

In Step 2, we have shown that there is at least one $\varphi(x) \in \omega(u)$ which is symmetric about the limiting plane $T_{\lambda_0^-}$. We denote one of them by $ \hat \varphi$ such that
$$\hat \psi _{\la_0^-}(x)=\hat \varphi (x^{\la_0^-})-\hat \varphi (x)\equiv 0.$$

Applying the same  method of moving planes starting from $\la$ near $+\infty$ and proceeding analogously as in the steps above, we obtain a  $\la_0^+ \geq \la_0^-$ and $\bar \varphi \in \omega(u)$ such that
$$\bar \psi_{\la_0^+}(x)=\bar \varphi (x^{\la_0^+})-\bar \varphi (x) \equiv0.$$

Now we prove that
\be \label{eq:4w4300}
\la_0^-=\la_0^+.
\ee
Step 3 will be completed once we show  \eqref{eq:4w4300}.  In fact, by the definition of $\la_0^-$ and $\la_0^+$, for each $\varphi\in \omega(u)$  we have
$$ \psi_{\la_0^-}(x)\geq 0,~x\in \Sigma_{\la_0^-}~\mbox{and}~\psi_{\la_0^+} (x) \leq 0,~x\in  \Sigma_{\la_0^+}.$$
If $ \la_0^-=\la_0^+$, then
$$\psi_{\la_0^-}(x)\equiv 0,~x\in \Sigma_{\la_0^-}.$$

For convenience, we turn to conduct the following process in the right region of the planes $T_\lambda$, notes $$ \tilde \Sigma_\lambda= \{x\in \mathbb R^N \mid x_1> \lambda \}.$$

We establish \eqref{eq:4w4300} by a contradiction argument. Suppose that \eqref{eq:4w4300} is not valid, then $\la_0^-<\la_0^+$,
and consequently for any $\la \in (\la_0^-,\la_0^+)$, we have
\be \label{eq:4w4301}
\hat \psi_\la(x)>0,~x\in \tilde \Sigma_\la
\ee
and
\be \label{eq:4w4302}
\bar \psi_\la(x)<0,~x\in \tilde \Sigma_\la.
\ee
Indeed, in the case $x\in \tilde \Sigma_{\la_0^+}$,  let $x^\lambda$ be the reflection of  $x$ about the plane $T_\lambda$ and $x_{\lambda_0^+}$ be the reflection of $x^\lambda$ about the plane $T_{\lambda_0^+}$, as one can see from Figure 2 below,
\begin{center}
\begin{tikzpicture}[scale=0.8]
\draw  [black] [->,thick](0,0)--(9.5,0) node [anchor=north west] {$x_1$};
\draw  [red] (3,-3)--(3,5) node [black][ above] {$T_{\lambda_0^-}$};
\draw  [black] (4,-3)--(4,5) node [black][ above] {$T_{\lambda}$};
\draw  [blue] (5,-3)--(5,5) node [black][ above] {$T_{\lambda_0^+}$};
\draw[red,domain=0.45:5.55] plot(\x,-\x*\x+6*\x-5) node at (1,2){$\hat \varphi(x)$};
\draw[blue,domain=2.45:7.55] plot(\x,-\x*\x+10*\x-21) node at (7.2,2){$\bar \varphi(x)$};
\path (5.5,3.75)[very thick,fill=black]  circle(1.4pt) node [ font=\fontsize{10}{10}\selectfont] at (6.55,3.9) {$(x,\bar \varphi(x))$};
\draw [dashed] (5.5,3.75)-- (5.5,0);
\path (5.5,0)[very thick,fill=black]  circle(1.4pt) node at (5.5,-0.2) {$x$};
\path (2.5,-2.25)[very thick,fill=black]  circle(1.3pt) node [ font=\fontsize{8}{8}\selectfont] at (1.55,-2.15) {$(x^\lambda,\bar \varphi(x^\la))$};
\draw [dashed] (2.5,-2.25)-- (2.5,0);
\path (2.5,0)[very thick,fill=black]  circle(1.4pt) node at (2.5,0.35) {$x^\la$};
\path (7.5,-2.25)[very thick,fill=black]  circle(1.3pt) node  [ font=\fontsize{8}{8}\selectfont] at (8.8,-2.25) {$(x_{\lambda_0^+},\bar \varphi(x_{\lambda_0^+}))$};
\draw [dashed] (7.5,-2.25)-- (7.5,0);
\path (7.5,0)[very thick,fill=black]  circle(1.4pt) node at (7.5,0.35) {$x_{\lambda_0^+}$};
\node [below=1cm, align=flush center,text width=8cm] at (4,-1.76)
        {$Figure$ 2 };
\end{tikzpicture}
\end{center}
since $ \bar \psi _{\la_0^+}\equiv 0$, we have
$$ \aligned
\bar \psi _\la(x)=\bar \varphi(x^\la)-\bar \varphi(x)=\bar \varphi(x_{\la_0^+})-\bar \varphi(x)<0,
\endaligned$$
due to the fact that  $\bar \varphi$ is decreasing in $x_1$ for $x_1>\la_0^+$ (similar to \eqref{weq:4w4444}).
In the case $x$ is between $T_\la$ and $T_{\la_0^+}$, we use the fact that $\bar \varphi$ is increasing in $x_1$ for $x_1<\la_0^+$.
Hence \eqref{eq:4w4302} holds for each $\la<\la_0^+$. Similarly, \eqref{eq:4w4301} holds for
each $\la>\la_0^-$.

For each compact subset $D\subset\subset \tilde \Sigma_{\la}$, by  \eqref{eq:4w4301} and \eqref{eq:4w4302}, there is a constant $q>0$ such that
\be \label{eq:4w4303}
\hat \psi _\la (x) >q,~x\in \bar D
\ee
and
\be \label{eq:4w4304}
\bar \psi _\la (x) <-q,~x\in \bar D.
\ee
Since $\hat \varphi, ~\bar \varphi \in \omega(u)$, there are sequences $\{t_n\}$ and $\{\bar t_n\}$ with $t_n<\bar t_n$ such that
$$u(\cdot, t_n)\rightarrow \hat \varphi(\cdot),~u(\cdot, \bar t_n)\rightarrow \bar \varphi (\cdot).$$
For sufficiently large $n$ we have, by  \eqref{eq:4w4303} and  \eqref{eq:4w4304},
\be \label{eq:4w4306}
w _\la (x, t_n) > {q} ,~x\in \bar D
\ee
and
$$w _\la (x,\bar t_n) <- {q} ,~x\in \bar D.$$

It follows that there exists $T_n\in(t_n,\bar t_n)$ such that
\be \label{eq:4w4307}\aligned
& w_\la(x,t)>0,~x\in \bar D,~ t\in [t_n,T_n),\\
& w_\la(\cdot, T_n) ~\mbox{vanishes~somewhere ~on }~\partial D.
\endaligned \ee

We will derive a contradiction with the second part of \eqref{eq:4w4307}. To this end, we first establish two major estimates.
\medskip

{\em (i) A global lower bond estimate:}

\be w_\la(x,t) \geq e^{-\theta(t-t_n)} \min\{0,   \underset{\tilde \Sigma_\la}{\inf}\ w_\la(x,t_n) \},~x\in \tilde \Sigma_\la,~ t\in [t_n,T_n] \label{A104} \ee while by our choice of $t_n$,
\be ~\underset{\tilde \Sigma_\la}{\inf} \  w_{\la}(x,  t_n) \rightarrow 0 ~\mbox{as}~  n\rightarrow \infty.
\label{A107}
\ee

{\em (ii) A positive lower bond estimate on a compact subset of $D$}.
\smallskip

There is a $D_0 \subset \subset D$ and a constant $C_0>0$ such that
$$w_\la(x,t) \geq  e^{-\theta(t-t_n)} C_0 ,~t_n\leq t\leq T_n, \, x \in D_0.$$

We first derive $(i)$.  By \eqref{eq:j1a}, similar to \eqref{17}, $w_\la(x,t)$ satisfies
\begin{equation*}
\left\{
\begin{array}{ll}
  \frac{\partial w_\la}{\partial t}+(-\lap)^s w_\la(x,t)=f_u(t,\xi_\la(x,t))w_\la(x,t),~&(x,t)\in \tilde \Sigma_\la \times (0,\infty),\\
  w_\lambda(x,t)=- w_\lambda(x^\lambda,t), ~& (x,t)     \in \tilde \Sigma_\la\times(0,\infty) ,
\end{array}
\right.
\end{equation*}
where $ \xi_\lambda(x,t)\ \text{is between}\ u_\lambda(x,t)\ \text{and}\ u(x,t)$.

Since $w_\la(\cdot,t)>0$ in $D$ for $t\in[t_n,T_n]$,    we just need to prove that
\be \label{eq:jj402w}
w_{\la}(x,t)\geq e^{-\theta(t-t_n)} \min\{0,\underset{\tilde \Sigma_\la}{\inf}   w_{\la}(x,  t_n)\},~x \in  \tilde \Sigma_\la \backslash D,~t\in [ t_n,T_n].
\ee
We choose $D$ to be $\tilde \Sigma_{\mu+\delta_0} \cap B_{\rho_2}(0)$ with $\mu > \la_0^+$ and $\rho_2>\rho_1$.  $\rho_1\geqslant R$ is selected by \eqref{21} so that
\be\label{32} f_u(t,\xi_\la(x,t)):=c_\la(x,t) < -\sigma,~x\in \mathbb R^N,~t>0,~|x|\geq \rho_1,\ee
while $\rho_2$ will be determined later (See Figure 3 below).
\begin{center}
\begin{tikzpicture}[scale=1]
\draw  [black] (1,-2.5)--(1,4.3) node [black][ above] {$T_{\lambda}$};
\draw  [blue] (1.5,-2.5)--(1.5,4.9) node [black][ above] {$T_{\lambda_0^+}$};
\draw  [red] (2,-2.5)--(2,4.3) node [black][ above] {$T_\mu$};
\draw  [black] (2.5,-2)--(2.5,4);
 \fill[blue!50] (2.5,-2) arc(-90:90:3) node [black] [above right] {$D=\tilde \Sigma_{\mu+\delta_0}\cap B_{\rho_2}(0)$};
 \draw (2.5,-1.6) arc(-90:90:2.6);
 \path (5.15,1.2 ) node [black] {$\textbf{D}$};
 \draw [black] (2.8,-2)--(2.8,4);
 \draw[black,thin,<-] (1,-1)--(2.8,-1);
 \draw[black,thin,->] (1,-1)--(2.8,-1);
 \path (1.8,-0.8) node [black]  {$d$};
 \draw[black,thin,<-] (2,3)--(2.5,3);
 \draw[black,thin,->] (2,3)--(2.5,3);
 \path (2.3,3.3) node [black]  {$\delta_0$};
 \node (test) at (5.7,2.2) {$\partial B_{\rho_1}$};
    \draw [->,thin] (4.6,2.5) to [in = 90, out = 30] (5.7,2.5) (test);
    \node (test) at (6.4,0.9) {$\partial B_{\rho_2}$};
    \draw [->,thin] (5.5,1.2) to [in = 90, out = 30] (6.4,1.2) (test);
  \node [below=1cm, align=flush center,text width=8cm] at (2,-1.5)
        {$Figure$ 3 };
\end{tikzpicture}
\end{center}
If we denote $$\big(\tilde \Sigma_\la \backslash B_{\rho_1}(0)\big)\cup
\Big((\tilde \Sigma_\la \backslash \tilde \Sigma_{\mu+\delta_0} )\cap B_{\rho_1}(0)\Big)=:E,$$  then
$$\tilde \Sigma_\la \backslash D\subset E.$$
Choose $\theta>0$ small, set $ \tilde w(x,t)=e^{\theta(t-t_n)} w_\la(x,t)$. In order to prove $(\ref{eq:jj402w})$, we turn to prove \be\label{38}
\tilde w_{\la}(x,t)\geq \min\{0,\underset{\tilde \Sigma_\la}{\inf} \tilde w_{\la}(x,  t_n)\},~x \in E,~t\in [ t_n,T_n].
\ee
Obviously, for $(x,t)\in \tilde \Sigma_\la\backslash E\times[ t_n,T_n]\subset D\times[ t_n,T_n]$,
$$
\tilde w_{\la}(x,t)\geq \min\{0,\underset{\tilde \Sigma_\la}{\inf} \tilde w_{\la}(x,  t_n)\}.
$$
So, by \eqref{32} and $f_u(t,\xi_\la(x,t))$ is bounded in $B_{\rho_1}(0)$,  applying  similar proof with  Theorem \ref{41a}, corresponding to
\eqref{31}, we have (\ref{38}). Thus we obtain \eqref{eq:jj402w}.

We also obtain if $\underset{\tilde \Sigma_\la}{\inf} ~  w_{\la}(x,  t_n)<0  $
$$ \underset{\tilde \Sigma_\la}{\inf} ~  w_{\la}(x,  t_n) \geq -\varepsilon_n \rightarrow 0,~\mbox{as}~n \rightarrow \infty,$$
where the convergence follows from $w_\la(\cdot,t_n) \rightarrow \hat \psi_\la (x)>0$. This proves $(i)$.

Then, we verify $(ii)$.  For that we need construct  a subsolution $ {\mathcal \zeta}(x,t)$ of $w_\la(x,t)$ in $D$.

Let $$ L_\la \zeta(x,t)=\frac{\partial \zeta}{\partial t}(x,t)+(-\lap)^s \zeta(x,t)-c_\la (x,t) \zeta(x,t).$$
We have $L_\la w_\la(x,t)=0,~x\in \tilde \Sigma_\la,~t>0.$
We want $\zeta(x,t)$ to satisfy
\be \label{eq:4wc2020}
\left\{\begin{array}{ll}
L_\la \zeta(x,t)\leq 0,& x\in D,~ t_n\leq t\leq T_n,\\
\zeta(x,t_n)\leq w_\la(x,t_n),& x\in D,\\
\zeta(x,t)\leq w_\la(x,t),&  x\in \tilde \Sigma_\la \backslash  D,~t_n\leq t\leq T_n.
\end{array}\right.
\ee
Here $t_n$ is sufficiently large.

Actually our final subsolution will be $\Psi (x,t) = C \zeta (x,t)$ for some suitable constant to be determined later.

For simplicity of notation, denote $\la=0$ and $T_\la=T_0$.
Let
\be\label{50}\zeta(x,t)=e^{-\theta(t-t_n)}\big(\tilde w_\mu(x,t)-\tau h(x)\big),\ee
where $\tau>0$ is to be determined,
fix $\theta$ satisfies $$
0<\theta<\sigma,
$$
$$\tilde w_\mu(x,t)=\left\{\begin{array}{ll}
w_\mu(x,t),& x_1>\mu,~t>t_n,\\
0, & -\mu\leq x_1\leq \mu,~t>t_n\\
w_\mu(x_1+2\mu,x',t),& x_1<-\mu,  ~t>t_n
\end{array}\right.$$
is a modification of $w_\mu(x,t)$ so that it is antisymmetric about $T_0$, and $$h(x)=\left\{\begin{array}{ll} 1, &  x_1\geq 0,\\
-1, & x_1<0,
\end{array}\right.$$
with the following two results hold:
\be\label{57}
(-\lap)^s \tilde w_\mu(x,t) -(-\lap)^sw_\mu (x,t) <0,\ \ x_1> \mu,\ t>0.\ee
and
\be \label{58}
(-\lap)^s h(x) =\frac{C}{x_1^{2s}},\ \ x_1>0,
\ee
where $ C$ is a positive constant.

See Lemma \ref{lem4w90} and \ref{lem4w91} in the Appendix for the proof of \eqref{57} and \eqref{58}.

We will verify that the $\zeta(x,t)$ defined in \eqref{50} satisfies \eqref{eq:4wc2020} by the the following three Lemmas.
\begin{lemma}\label{lem4w92}(The differential inequality)
For $ D$ of the shape $\tilde \Sigma_{\mu+\delta_0} \cap B_{\rho_2}(0)$, we have
$$L_\la \zeta(x,t)\leq 0,~\forall~x\in D,~t_n\leq t \leq T_n.$$
\end{lemma}
\begin{proof}[\bf Proof] Noting that in $D$ we have $\tilde w_\mu (x,t)=w_\mu(x,t)$, then by \eqref{57}, for $x_1>\mu+\delta_0$, we have
$$\aligned &L_\la \zeta(x,t)\\
=&\frac{\partial \zeta}{\partial t}(x,t)+(-\lap)^s \zeta(x,t)-c_\la (x,t) \zeta(x,t) \\
=& -\theta e^{-\theta(t-t_n)}\big(\tilde w_\mu(x,t) -\tau h(x)\big)+e^{-\theta(t-t_n)} \frac{\partial \tilde w_\mu}{\partial t}(x,t)\\&
+ e^{-\theta(t-t_n)}\big((-\lap)^s\tilde w_\mu(x,t) -\tau(-\lap)^sh(x)\big)- e^{-\theta(t-t_n)}c_\la(x,t)\big(\tilde w_\mu(x,t)-\tau h(x)\big)\\
=& e^{-\theta(t-t_n)}\Big\{-\theta w_\mu(x,t)+\theta\tau +\frac{\partial w_\mu}{\partial t}(x,t) +[(-\lap)^s\tilde w_\mu-(-\lap)^s w_\mu(x,t)]\\&+(-\lap)^s w_\mu(x,t)-\tau (-\lap)^s h(x)
-c_\la(x,t) w_\mu(x,t)+\tau c_\la(x,t)\Big\}\\
=& e^{-\theta(t-t_n)}\Big\{-\theta w_\mu(x,t)+\theta\tau +[(-\lap)^s\tilde w_\mu(x,t)-(-\lap)^s w_\mu(x,t)]\\
&-\tau (-\lap)^s h(x)+\big(c_\mu(x,t)-c_\la(x,t)\big)w_\mu(x,t)+\tau c_\la(x,t)\Big\}\\
< & e^{-\theta(t-t_n)}\Big\{[-\theta +c_\mu(x,t)-c_\la(x,t)]w_\mu(x,t)+\tau\big(\theta+c_\la(x,t)\big)-\tau (-\lap)^sh(x) \Big\}.
\endaligned$$
By \eqref{58}, there exists $d>\lambda_0^+$, such that for $0<x_1<d$,
$$\big(\theta+c_\la(x,t)\big)-(-\lap)^s h(x)\leq 0.$$
Choose $\mu$ such that $\lambda_0^+<\mu<d-\delta_0$ sufficiently close to $\la\ ($where $\delta_0>0$ small will be chosen later$)$, since $f_u(\cdot,u)$ is continuous in $u$ and thus in $x_1$, we have
$$  c_\mu(x,t)-c_\la(x,t)\leq \frac{\theta}{2}.$$
Hence
$$L_\la \zeta(x,t)\leq e^{-\theta(t-t_n)}[-\frac{\theta}{2}w_\mu(x,t)+\tau \big(\theta+c_\la(x,t)\big)-\tau(-\lap)^s h(x)].$$
That is, we obtain
\be\label{53}
L_\la \zeta(x,t)\leq0,\ \ ~\mu<x_1<d,~t_n\leq t \leq T_n.
\ee
By \eqref{32} and the the choice of $\theta$, We have for $|x|\geq \rho_1$,
$$\theta+c_\la(x,t)\leq 0.$$
So, for $x\in \tilde \Sigma_{\la+d} \cap B_{\rho_1}^c(0)\cap D $, we have
\be\label{54}
L_\la \zeta(x,t)\leq0, \ \ ~t_n\leq t \leq T_n.
\ee
Now for $x\in\tilde \Sigma_{\la+d} \cap B_{\rho_1}(0)\cap D $, there exists a constant $c_0>0$, such that
$$ w_\mu(x,t)\geq c_0,~\forall~t_n\leq t\leq T_n.$$
So, we can choose $\tau$ sufficiently small, such that
$$ -\frac{\theta}{2}w_\mu (x,t)+\tau \big(\theta+c_\la(x,t)\big) \leq 0.$$
As a result,
\be\label{55}
L_\la \zeta(x,t)\leq0, ~x\in \tilde \Sigma_{\la+d}\cap B_{\rho_1}(0) \cap D,~t_n\leq t\leq T_n.\ee
Combining \eqref{53}, \eqref{54} and \eqref{55}, we completes the proof of Lemma \ref{lem4w92}.
\end{proof}
\begin{remark}

In this step, we will also choose $\tau$ small to satisfy
\be c_0 - \tau \geq a_0
\label{A109}
\ee
for some positive constant $a_0$, which will be used later.
\end{remark}

\begin{lemma}\label{lem4w94}(The initial condition)
From \eqref{eq:4w4306}, we have $w_\la(x,t_n)> q$ (independent of large $n$). Denote
$$ \Psi(x,t)=q \frac{\zeta(x,t)}{\|\zeta(x,t_n)\|_{L^\infty(D)}}.$$
Then
$$ w_\la(x,t_n)\geq \Psi(x,t_n), \;\; \forall x \,\in D.$$
\end{lemma}

The proof of this Lemma is trivial, and we skip it.

\begin{remark}
Obviously, this $\Psi(x,t)$ still satisfies the differential inequality
 $$L_\la \Psi(x,t)\leq 0.$$
\end{remark}

\begin{lemma}\label{lem4w93}(The exterior condition)
For $x\in   \tilde \Sigma_\la \backslash D,~t_n \leq t\leq T_n,$ we have $$ w_\la(x,t)\geq \Psi(x,t).$$
\end{lemma}
\begin{proof}[\bf Proof.]Denote $\gamma_n = \|\zeta(x,t_n)\|_{L^\infty(D)}$. From the expression of $\zeta (x, t)$, one can see that there exist $M>0$, such that
\begin{equation}  \gamma_n \leq M.
\label{A100}
\end{equation}

For any given $\rho_2 > 0$ (to be determined soon), divide $\tilde \Sigma_\la \backslash D$ into two parts:
$$D_1^C \equiv \{ x \in \tilde \Sigma_\la \backslash D \mid |x| > \rho_2 \} ~{\mbox{ and }}~ D_2^C \equiv \{ x \mid \lambda < x_1 \leq \mu + \delta_0, |x| \leq \rho_2 \}.$$

$a)$ In $D_1^C$,
since $\tilde{w}_\mu(x,t)\rightarrow 0$  as $|x|\rightarrow \infty$ uniformly for large $t$, then for sufficiently large $\rho_2>\rho_1$
as chosen in $(i)$, we have
\begin{equation} \tilde{w}_\mu(x,t) \leq   \frac{\tau}{2} \equiv \frac{\tau}{2} h(x), \;\;\; |x| \geq \rho_2, \;\;\;  ~{\mbox{ for all large t}}~ .
\label{A101}
\end{equation}

$b)$ In $D_2^C$, we have $\tilde w_\mu(x,t)=0,~\la <x_1<\mu $  and
$$\tilde w_\mu(x,t)=w_\mu(x,t),~\mu\leq x_1\leq \mu+\delta_0.$$
Since $w_\mu(x,t)$ is uniformly Lipschitz (from such property of $u$), there is $C_2 >0$, such that
$$ |w_\mu(x,t)|\leq C_2(x_1-\mu).$$
Choosing $\delta_0$ small, such that for $\mu<x_1<\mu+\delta_0$
$$C_2(x_1-\mu)<\frac{\tau}{2} \equiv \frac{\tau}{2} h(x).$$

It follows from this and (\ref{A101}),\ (\ref{A100}), for all $x \in \tilde \Sigma_\la \backslash D$,
\be \label{eq:cw2008}
\Psi(x,t) \leq - e^{-\theta (t-t_n)} \frac{q}{\gamma_n} \cdot \frac{\tau}{2} \leq - e^{-\theta (t-t_n)} \frac{q}{M} \cdot \frac{\tau}{2}
\ee

On the other hand,  by (\ref{A104}) and (\ref{A107}),
$$ w_\la(x,t)\geq -e^{-\theta(t-t_n)}\varepsilon_n,$$
$$\mbox{with}~ \varepsilon_n\rightarrow 0,~\mbox{as}~t_n\rightarrow \infty,~(-\varepsilon_n \leq \underset{\tilde \Sigma_\la}{\inf}w_\la(x,t_n)).$$
Then by (\ref{eq:cw2008}), for sufficiently large $n$, we have, for all $x \in \tilde \Sigma_\la \backslash D$,
\begin{equation*}
w_\la(x,t)- \Psi(x,t) \geq e^{-\theta(t-t_n)}[ -\varepsilon_n + \frac{q}{M} \cdot \frac{\tau}{2}] \geq 0.
\end{equation*}
This completes the proof of the Lemma \ref{lem4w93}.
\end{proof}
\medskip

Now we use $\Psi(x,t)$ as our subsolution.  Combining Lemmas \ref{lem4w92}, \ref{lem4w93}, \ref{lem4w94} and the {\em maximum principle for anti-symmetric functions} (Theorem \ref{thm4w98}),  we have
$$ w_\la(x,t)\geq \Psi(x,t),~x\in D,~t_n\leq t\leq T_n.$$
Choosing $D_0 \equiv \tilde \Sigma_{\la+d} \cap B_{\rho_1}(0)\cap D \subset\subset D$, then by (\ref{A109}), we have
\be \label{eq:wcw2033}
e^{\theta(t-t_n)}\zeta(x,t)=w_\mu(x,t)-\tau h(x) >a_0,~~x\in D_0,~t \in [t_n, T_n].
\ee

Then by the definition of $\Psi (x,t)$ and (\ref{A100}),
we have
$$ \Psi(x,t) \geq e^{-\theta(t-t_n)} \frac{q a_0}{M}, \;\; x \in D_0,\ t \in [t_n, T_n].$$
Consequently
\be
w_\lambda (x, t) \geq e^{-\theta(t-t_n)} \frac{q a_0}{M} \equiv e^{-\theta(t-t_n)} C_0, \;\; x \in D_0, t \in [t_n, T_n].
\label{A121}
\ee

This verifies {\em Estimate (ii)}.
\medskip

Now we use {\em Estimates (i)} and {\em (ii)}  to derive a contradiction with the second part of (\ref{eq:4w4307}), that is
$$ w_\la(\cdot, T_n) ~\mbox{vanishes~somewhere ~on }~\partial D.$$

Suppose for $\bar{x} \in \partial D$,
$$w_\lambda (\bar{x}, T_n) = 0.$$

Let $D_0$ be the compact subset of $D$ given in (\ref{A121}).

Let $\delta >0$ be small so that
\be
B_\delta (\bar{x}) \cap D_0 = \emptyset \mbox{ and } w_\lambda(x, t_n) > \frac{q}{2}, \;\; x \in B_\delta (\bar{x}).
\label{A130}
\ee

We will construct a sub-solution $\underline{w}(x,t)$ in $B_\delta(\bar{x}) \times [t_n, T_n]$, so that $\underline{w}(\bar{x}, T_n)>0$ to derive a contradiction.

First we modify $w_\lambda$. Let
$$\tilde{w}(x,t) = e^{-m(t-t_n)}w_\lambda (x,t).$$
Then
$$\tilde{L} \tilde{w} \equiv \frac{\partial \tilde{w}}{\partial t} + (-\lap)^s \tilde{w} - \tilde{C}(x,t) \tilde{w} = 0,$$
where
$$ \tilde{C}(x,t) = c_\lambda (x,t) -m.$$
Choose $m>0$, so that
$$ \tilde{C}(x,t) <0.$$

Let $$\phi(x) = (1-|x|^2)^s_+ \; \mbox{ and } \phi_\delta(x) = \phi\left(\frac{x-\bar{x}}{\delta}\right).$$
Then it is well-known that
$$(-\lap)^s \phi_\delta (x) = a_\delta, \;\;  x \in B_\delta (\bar{x})$$
for some constant $a_\delta$ depending on $\delta$.

Let $$\underline{w} (x,t) = \chi_{D_0}(x) \tilde{w}(x,t) + (2\phi_\delta(x) -1)\varepsilon_n e^{-(m+\theta)(t-t_n)},$$
where $$  \chi_{D_0}(x) = \left\{\begin{array}{ll}
1 & x \in D_0\\
0 & x \not\in D_0.
\end{array}
\right.
$$

We verify that $\underline{w} (x,t)$ is a sub-solution of $\tilde{w}(x,t)$ in the parabolic cylinder
$$B_\delta (\bar{x}) \times [t_n, T_n]. $$

First consider the initial condition at $t=t_n$ and $x \in B_\delta(\bar{x})$. We have
$$\underline{w} (x,t_n) = (2\phi_\delta(x) -1)\varepsilon_n \leq \varepsilon_n, $$
while by (\ref{A130}),
$$\tilde{w}(x,t_n) = w_\lambda (x, t_n) > \frac{q}{2}.$$
We choose $n$ sufficiently large, so that $\varepsilon_n < \frac{q}{2}$ to satisfy the initial condition.

Then we check the exterior condition in $(\tilde{\Sigma}_\lambda \setminus B_\delta (\bar{x})) \times [t_n, T_n]$.

For $x \in D_0$,
$$ \underline{w} (x,t) = \tilde{w}(x,t) - \varepsilon_n e^{-(m+\theta)(t-t_n)} \leq \tilde{w}(x,t), \; t \in [t_n, T_n],$$
while for $x \in \tilde{\Sigma}_\lambda \setminus D_0$ and $x \not \in B_\delta(\bar{x})$,
$$\underline{w}(x,t) = - \varepsilon_n e^{-(m+\theta)(t-t_n)} \leq e^{-m(t-t_n)}w_\lambda(x,t) = \tilde{w}(x,t), \; t \in [t_n, T_n].$$
This verifies the exterior condition.

Finally, we deduce the differential inequality
$$\tilde{L} \underline{w} (x,t) \leq 0, \;\; (x,t) \in B_\delta(\bar{x}) \times (t_n, T_n].$$

In fact, for $(x,t) \in B_\delta(\bar{x}) \times (t_n, T_n]$,
\begin{eqnarray} \tilde{L} \underline{w}(x,t) &=& -(m+\theta) (2\phi_\delta(x) -1) \varepsilon_n e^{-(m+\theta)(t-t_n)} + (-\lap)^s (\chi_{D_0}(x) \tilde{w}(x,t)) \nonumber \\
&+& 2 a_\delta \varepsilon_n e^{-(m+\theta)(t-t_n)} -\tilde{C}(x,t) (2\phi_\delta(x) -1)\varepsilon_n e^{-(m+\theta)(t-t_n)} \nonumber\\
&\leq& (-\lap)^s (\chi_{D_0}(x) \tilde{w}(x,t)) + C_1 \varepsilon_n e^{-(m+\theta)(t-t_n)}, \label{A132}
\end{eqnarray}
for some constant $C_1$ independent of $n$.

For each fixed $t \in [t_n, T_n]$ and for $x \in B_\delta (\bar{x})$, by (\ref{A121}),
\begin{eqnarray}
(-\lap)^s (\chi_{D_0}(x) \tilde{w}(x,t)) &=& C_{N,s} \int_{D_0} \frac{-\tilde{w}(y,t))}{|x-y|^{N+2s}} d y \nonumber \\
&\leq& - e^{-(m+\theta)(t-t_n)} C_0 C_{N,s} \int_{D_0} \frac{1}{|x-y|^{N+2s}} d y \nonumber \\
&\leq& - C_2 e^{-(m+\theta)(t-t_n)},
\label{A133}
\end{eqnarray}
with some positive constant $C_2$ independent of $n$.

For all  sufficiently large $n$, we have $\varepsilon_n C_1 \leq C_2$, and hence by (\ref{A132}) and (\ref{A133}),
$$ \tilde{L} \underline{w}(x,t) \leq 0 = \tilde{L} \tilde{w} (x,t), \;\; (x,t) \in B_\delta (\bar{x}) \times [t_n, T_n]. $$

Now by the {\em maximum principle}, we conclude
$$ \underline{w}(x,t) \leq  \tilde{w} (x,t), \;\; (x,t) \in B_\delta (\bar{x}) \times [t_n, T_n]. $$
In particular
$$  \tilde{w} (\bar{x},T_n) \geq \underline{w}(\bar{x},T_n) = \varepsilon_n e^{-(m+\theta)(T_n-t_n)}.$$
Consequently
$$ w_\lambda (\bar{x}, T_n) \geq \varepsilon_n e^{-\theta(T_n-t_n)} >0.$$
This contradicts our assumption that
$$ w_\lambda (\bar{x}, T_n) =0$$
and thus completes the proof of \eqref{eq:4w4300}.

Since the $x_1$-direction can be chosen arbitrarily, $\psi_{\la_0^-}(x)\equiv 0$ for all $\varphi \in \omega(u)$ means that all $\varphi$ are    radially symmetric and monotone decreasing about some point in $\mathbb R^N$.

\end{proof}

\section{Appendix}
\begin{lemma}\label{mlem1}
For any $x\in B_\delta(\bar{x})\times[0,2]$, $g(x)=(\delta^2-|x-\bar{x}|^2)_+^{s}-(\delta^2-|x-\bar{x}^\lambda|^2)_+^{s}$, we have
$$|(-\Delta)^sg(x)|\leq C_0,$$ where $C_0$ is a constant.
\end{lemma}
\begin{proof}[\bf Proof]
Let $\phi(x)=c(1-|x|^2)_+^s$. By choosing suitable $c$,  we  have
\begin{equation*}
\left\{\aligned
&(-\Delta)^s\phi(x)=1,  & x \in  B_1(0),\\
&\phi(x)=0,  & x \in B_1^c(0).
\endaligned \right.
\end{equation*}
Let $\phi_\delta(x)=c(\delta^{2}-|x-\bar{x}|^2)_+^s$ and $\phi_\delta^\lambda(x)=c(\delta^{2}-|x-\bar{x}^\lambda|^2)_+^s$.
It is easy to check that
\begin{equation}\label{lmeq1}
\left\{\aligned
&(-\Delta)^s\phi_\delta(x)=1,  &x \in  B_\delta(\bar{x}),\\
&\phi_\delta(x)=0,  &x \in B_\delta^c(\bar{x}).
\endaligned \right.
\end{equation}
and
\begin{equation}\label{lmeq2}
\left\{\aligned
&(-\Delta)^s\phi_\delta^\lambda(x)=1,  & x \in  B_\delta(\bar{x}^\lambda),\\
&\phi_\delta^\lambda(x)=0,  &x \in  B_\delta^c(\bar{x}^\lambda).
\endaligned \right.
\end{equation}
Since $g(x)=\frac{1}{c}\big(\phi_\delta(x)-\phi_\delta^\lambda(x)\big)$, by \eqref{lmeq1} and \eqref{lmeq2}, we have
$$|(-\Delta)^sg(x)|=\frac{1}{c}|(-\Delta)^s\phi_\delta(x)-(-\Delta)^s\phi_\delta^\lambda(x)|\leq C_0.$$
\end{proof}
\begin{lemma}\label{lem4w90}
For $x_1> \mu,\ t>0,$ we have
$$(-\lap)^s \tilde w_\mu(x,t) -(-\lap)^sw_\mu (x,t) <0.$$
\end{lemma}

\begin{proof}[\bf Proof] For $x_1>\mu$, for each $t>0$, noting that $\tilde w_\mu(x,t)=w_\mu(x,t)$, by the definition of $\tilde w_\mu$, we obtain
\begin{eqnarray*}
&&(-\lap)^s \tilde w_\mu(x,t) -(-\lap)^s w_\mu(x,t)\\
&=&C_{N,s}P.V. \int_{\mathbb R^{N}} \frac{w_\mu(x,t)-\tilde w_\mu(y,t)}{|x-y|^{N+2s}}dy-C_{N,s}P.V. \int_{\mathbb R^{N}} \frac{w_\mu(x,t)-w_\mu(y,t)}{|x-y|^{N+2s}}dy\\
&=&C_{N,s}P.V.  \int_{\mathbb R^{N}} \frac{w_\mu(y,t)-\tilde w_\mu(y,t)}{|x-y|^{N+2s}}dy \\
&=&C_{N,s}\int_{-\infty}^{\mu} \int_{\mathbb R^{N-1}} \frac{w_\mu(y,t)}{|x-y|^{N+2s}}dy'dy_1
-C_{N,s}\int_{-\infty}^{-\mu} \int_{\mathbb R^{N-1}}\frac{\tilde w_\mu(y,t)}{|x-y|^{N+2s}}dy'dy_1 \\
&=&C_{N,s} \int_{\mathbb R^{N-1}} \Big(\int_{-\infty}^{\mu}\frac{w_\mu(y,t)}{(|x_1-y_1|^2+|x'-y'|^2)^{\frac{N+2s}{2}}}\nonumber\\
&&-\int_{-\infty}^{-\mu} \frac{w_\mu(y_1+2\mu,y',t)}{(|x_1-y_1|^2+|x'-y'|^2)^{\frac{N+2s}{2}}}\Big)dy_1dy'\\
&=&C_{N,s}\int_{-\infty}^{\mu} \int_{\mathbb R^{N-1}} \Big(\frac{w_\mu(y,t)}{(|x_1-y_1|^2+|x'-y'|^2)^{\frac{N+2s}{2}}}-\nonumber\\
&&\frac{w_\mu(y_1,y',t)}{(|x_1-y_1+2\mu|^2+|x'-y'|^2)^{\frac{N+2s}{2}}}\Big)dy'dy_1<0,
\end{eqnarray*}
where the last inequality holds due to
$w_\mu(y,t)<0$ for $y_1<\mu$ and
$$\frac{1}{(|x_1-y_1|^2+|x'-y'|^2)^{\frac{N+2s}{2}}}>\frac{1}{(|x_1-y_1+2\mu|^2+|x'-y'|^2)^{\frac{N+2s}{2}}}.$$
\end{proof}

\begin{lemma}\label{lem4w91}
For $x_1>0$, we have
\be \label{4w2021}
(-\lap)^s h(x) =\frac{C}{x_1^{2s}},
\ee
where $ C$ is a positive constant.
\end{lemma}
\begin{proof}[\bf Proof ]
For $x_1>0$, by the definition of $h(x)$, we have
$$\aligned (-\lap)^s h(x)=&C_{N,s} \int_{\mathbb R^{N}} \frac{1-h(y_1)}{|x-y|^{1+2s}}dy\\
=& C_{N,s}P.V. \int_{\mathbb R^{N}} \frac{1-h(y_1)}{(|x_1-y_1|^2+|x'-y'|^2)^{\frac{N+2s}{2}}}dy'dy_1,\ \text{letting }y'-x'=|x_1-y_1|z'\\
=& C_{N,s}P.V. \int_{\mathbb R}\int_{\mathbb R^{N-1}} \frac{|x_1-y_1|^{N-1}(1-h(y_1))}{|x_1-y_1|^{N+2s}(1+|z'|^2)^{\frac{N+2s}{2}}}dz'dy_1\\
=&\bar{C}  \int_{\mathbb R}\frac{1-h(y_1)}{|x_1-y_1|^{1+2s}}dy_1
=\bar{C}  \int_{-\infty}^0 \frac{1-h(y_1)}{|x_1-y_1|^{1+2s}}dy_1,\ \ \ \text{letting }y_1=x_1z_1\\
=& \frac{2\bar{C}}{x_1^{2s}}\int_{-\infty}^0 \frac{1}{|1-z_1|^{1+2s}}dz_1=\frac{C}{x_1^{2s}}.
\endaligned $$
\end {proof}

\end{document}